\documentclass[11pt]{amsart}
\usepackage{amssymb,amscd}

\DeclareMathOperator{\Spec}{Spec}
\DeclareMathOperator{\iz}{int}

\newcommand{\ra}{\rightarrow}
\newcommand{\ssm}{\smallsetminus}
\begin{document}

\newtheorem{thm}{Theorem}[section]
\newtheorem{lem}[thm]{Lemma}
\newtheorem{prop}[thm]{Proposition}
\newtheorem{cor}[thm]{Corollary}
\newtheorem{ex}[thm]{Example}
\newtheorem{rem}[thm]{Remark}
\newtheorem{prob}[thm]{Problem}
\newtheorem{nota}[thm]{Notation}

\newtheorem{thmA}{Theorem}
\renewcommand{\thethmA}{}

\theoremstyle{definition}

\newtheorem{defi}[thm]{Definition}
\renewcommand{\thedefi}{}

\input amssym.def

\long\def\alert#1{\smallskip{\hskip\parindent\vrule%
\vbox{\advance\hsize-2\parindent\hrule\smallskip\parindent.4\parindent%
\narrower\noindent#1\smallskip\hrule}\vrule\hfill}\smallskip}

\def\ff{\frak}
\def\Spec{\mbox{\rm Spec}}
\def\type{\mbox{ type}}
\def\Hom{\mbox{ Hom}}
\def\rk{\rm{rk}}
\def\Ext{\mbox{ Ext}}
\def\Ker{\mbox{ Ker}}
\def\Max{\mbox{\rm Max}}

\def\End{\mbox{\rm End}}
\def\l{\langle\:}
\def\r{\:\rangle}
\def\Rad{\mbox{\rm Rad}}
\def\Zar{\mbox{\rm Zar}}
\def\Supp{\mbox{\rm Supp}}
\def\Rep{\mbox{\rm Rep}}
\def\Inv{{\rm Inv}}
\def\Div{{\rm Div}}
\def\Prin{{\rm Prin}}
\def\cal{\mathcal}
\def\O{{\cal O}}
\def\Ms{{\cal M}_{\#}}
\def\Md{{\cal M}_{\dagger}}
\newcommand\sinv[1]{#1\raisebox{1.15ex}{$\scriptscriptstyle-\!1$}}
\def\Maxi{\sinv{\mbox{\rm Max}}}

\newcommand{\Z}{\mathbb{Z}}
\newcommand{\N}{\mathbb{N}}

\title[Group-theoretic and topological invariants]{Group-theoretic and topological invariants of
completely integrally closed Pr\"ufer domains}

\thanks{2010 {\it Mathematics Subject Classification.}
Primary 13F05, 13A15, 13A05; Secondary 06F15.}
\thanks{\today}

\author{Olivier A.~Heubo-Kwegna,  Bruce Olberding and Andreas Reinhart}

\address{Department of Mathematical Sciences, Saginaw Valley State University,
7400 Bay Road, University Center, MI 48710-0001}
\email{oheubokw@svsu.edu}

%\address{Department of Mathematics and Statistics, Bowling Green State University, Bowling Green,
 %OH, 43403}
%\email{warrenb@bgnet.bgsu.edu}

%\address{H. L. Wilkes Honors College, Florida Atlantic University,
% Jupiter, Fl, 33458}
%\email{warren.mcgovern@fau.edu}

\address{Department of Mathematical Sciences, New Mexico State University, Las Cruces, NM 88003}
\email{olberdin@nmsu.edu}

\address{Institut f\"ur Mathematik und wissenschaftliches Rechnen, Karl-Franzens-Universit\"at Graz, NAWI Graz, Heinrichstrasse 36, 8010 Graz, Austria}
\email{andreas.reinhart@uni-graz.at}

\begin{abstract}  We consider the lattice-ordered groups $\Inv(R)$ and $\Div(R)$ of invertible and divisorial fractional ideals  of a completely integrally closed Pr\"ufer domain. We prove that $\Div(R)$ is the completion of the group $\Inv(R)$, and we show there is a faithfully flat extension $S$ of $R$ such that   $S$ is a completely integrally closed B\'ezout domain with $\Div(R) \cong \Inv(S)$.  Among the class of completely integrally closed Pr\"ufer domains, we focus on  the one-dimensional Pr\"ufer domains.
This class  includes  Dedekind domains, the latter being the one-dimensional Pr\"ufer domains  whose maximal ideals are finitely generated.  However, numerous interesting examples
 show that the class of one-dimensional Pr\"ufer domains
includes domains that differ quite significantly from Dedekind domains
  by a number of measures, both group-theoretic (involving $\Inv(R)$ and $\Div(R)$) and topological (involving the maximal spectrum of $R$).  We examine these invariants in connection with factorization properties of the ideals of one-dimensional Pr\"ufer domains, putting special emphasis on the class of almost Dedekind domains, those domains for which every localization at a maximal ideal is a rank one discrete valuation domain, as well as the class of SP-domains, those domains for which every proper ideal is a product of radical ideals. For this last class of domains, we show that if in addition the ring has nonzero Jacobson radical, then  the lattice-ordered groups $\Inv(R)$ and $\Div(R)$ are determined entirely by the topology of the maximal spectrum of $R$, and that the Cantor-Bendixson derivatives of the maximal spectrum reflect  the distribution of sharp and dull maximal ideals.
\vspace{0.50in}

{\noindent}Key words: Pr\"ufer domain, almost Dedekind domain, SP-domain, $\ell$-group, Boolean space, Cantor-Bendixson derivative.
\end{abstract}

\maketitle

\section{Introduction}

All rings considered in this article are commutative with unity.
An integral domain $R$ with quotient field $F$ is {\it completely integrally closed} if for each $x \in F \setminus R$ and $0 \ne r \in R$, there exists $n >0$ such that $rx^n \not \in R$.
 A {\it Pr\"ufer domain} $R$ is an  integral domain for which every nonzero finitely generated ideal is invertible; equivalently, for each maximal ideal $M$ of $R$, $R_M$ is a valuation domain.
 If in addition every invertible ideal is principal, then  $R$ is a {\it B\'ezout domain}.
  The class of completely integrally closed Pr\"ufer domains includes a number of prominent examples in non-Noetherian commutative ring theory, such as the ring of integer-valued polynomials \cite[Proposition VI.2.1, p.~129]{CC}, the ring of entire functions \cite[Proposition 8.1.1(6), p.~276]{FHP}, the real holomorphy ring of a function field \cite[Corollary 3.6]{OlbH}, and the Kronecker function ring of a field extension of at most countable transcendence degree \cite[Corollary 3.9]{Heu}.  For a completely integrally closed Pr\"ufer domain $R$, the set $\Div(R)$ of nonzero divisorial fractional ideals of $R$ is a lattice-ordered group ($\ell$-group) with respect to ideal multiplication, as is the group $\Inv(R)$ of invertible fractional ideals of $R$.   We prove in Proposition~\ref{completely integrally prufer}
   that $\Div(R)$ is the  completion of $\Inv(R)$, a consequence of which is that  the group $\Div(R)$ can be calculated solely from the $\ell$-group $\Inv(R)$.
 We  show also in Theorem~\ref{dense subring} that for each completely integrally closed Pr\"ufer domain $R$, there exists a faithfully flat extension $S$ of $R$ such that $S$ is a completely integrally closed B\'ezout domain with  $\Div(S) = \Inv(S) \cong \Inv(R)$. For this ring $S$, every divisorial ideal is principal.

 After treating the general case of completely integrally closed Pr\"ufer domains in Section 3, we focus for the rest of the article on the special case of one-dimensional Pr\"ufer domains. While this class includes the class of Dedekind domains,
 even among the class of one-dimensional Pr\"ufer domains, the Dedekind domains are quite special.
For example, consider  a nonzero proper ideal $I$ of a one-dimensional Pr\"ufer domain $R$. If $R$ is a Dedekind domain, then $\Spec(R/I)$ is a finite set. If $R$ is not necessarily Dedekind,
it can only be asserted that  $\Spec(R/I)$ is a {\it Boolean space}, that is, a compact Hausdorff zero-dimensional space \cite[p.~198]{Joh}. It follows from \cite[Section 3]{Olb} that every Boolean space arises as $\Spec(R/I)$ for an appropriate choice of one-dimensional Pr\"ufer domain $R$ with nonzero Jacobson radical $I$.  By Stone Duality, the category of Boolean spaces is dual to the category of Boolean algebras. Thus the class of one-dimensional Pr\"ufer domains is at least as rich and varied as the category of Boolean algebras.

In fact, the construction in \cite[Section 3]{Olb} shows that each Boolean space can be realized in this fashion by an almost Dedekind domain $R$ with nonzero Jacobson radical. (A domain $R$ is {\it almost Dedekind} if $R_M$ is a rank one discrete valuation ring (DVR) for each maximal ideal $M$ of $R$.) The ring $R$ in this case is even an {\it SP-domain}, a domain for which every proper ideal is a product of radical ideals. The isolated points in the maximal spectrum $\Max(R)$ of $R$ are precisely the finitely generated (hence invertible) maximal ideals of $R$ \cite[Lemma 3.1]{Olb}. Thus, if $\Max(R)$ is infinite and the Jacobson radical of $R$ is nonzero, then the compactness of $\Max(R)$ implies that  $R$ will have maximal ideals that are not finitely generated.  Because a great deal of variation is possible in the distribution of isolated points in a Boolean space, it follows that
 even among the almost Dedekind domains for which every proper ideal is a product of radical ideals,
  a wide range of  behavior with respect to finite generation of maximal ideals is possible.

The aforementioned existence results rely on the Krull-Kaplansky-Jaffard-Ohm Theorem to realize a certain lattice-ordered group as the group of divisibility of a {\it B\'ezout domain}, that is, a domain for which every finitely generated ideal is principal. As such they give   existence in what might be considered the somewhat exotic setting of overrings of polynomial rings with as many variables as there are elements of the group. However, much of this same wide range of behavior can be seen in number-theoretic inspired contexts involving infinite towers of finite field extensions \cite{BY, GR, Gil2, Has, Lop, LL, REK, VY,Y}. Interesting almost Dedekind domains also appear in the context of holomorphy rings   \cite{OlbMat}.

Motivated by the diversity of  these examples,  we focus on (1) the group-theoretic invariants $\Inv(R)$ and $\Div(R)$, which encode much of the multiplicative ideal theory of the domain, and (2) the topology of the maximal spectrum $\Max(R)$, which encodes information on the density of finitely generated maximal ideals among the collection of all maximal ideals of  $R$.  %The group $\Inv(R)$ plays for a Pr\"ufer domain a similar role to the well-studied  group of divisibility of a B\'ezout domain; see for example \cite{BK}.
We are especially interested in the case in which $R$ is an SP-domain with nonzero Jacobson radical, and in this case the group $\Inv(R)$ is entirely determined by $\Max(R)$ and vice versa (see Theorem~\ref{characterization}).  In this sense, much of the multiplicative ideal theory of  the domain $R$ is accounted for simply from the topology of  $\Max(R)$ alone.  For example, not only is $\Inv(R)$  determined by $\Max(R)$, the group $\Div(R)$ of nonzero fractional  divisorial ideals of $R$ can also be extracted from the topology of $\Max(R)$ by using the Gleason cover of the topological space $\Max(R)$ (Theorem~\ref{EX cor}).
Similarly,  we show in Section 6 how to reinterpret Loper and Lucas' notions of sharp and dull degrees from \cite{LL} for one-dimensional Pr\"ufer domains in purely topological terms. In doing so, we obtain some of their existence results as simple consequences of well-known topological results, and we extend these results from finite numbers to infinite ordinals, giving another measure of complexity of this class of rings.

The outline of the paper is as follows. After a brief discussion of background material in Section 2,
we develop in Section~3 a  multiplicative ``completion'' of a completely integrally closed Pr\"ufer domain by showing that for such a domain $R$, the group $\Div(R)$ is the lattice-ordered group completion of $\Inv(R)$ and  can be realized as $\Inv(S)$ for a B\'ezout domain $S$ extending $R$ (Theorem~\ref{dense subring}).

 In Section 4 we characterize SP-domains with emphasis on principal rather than invertible ideals, and we note in Theorem~\ref{SPG3} that if a one-dimensional domain $R$ has nonzero Jacobson radical $J(R)$ such that $J(R)$ is invertible, then $R$ is an SP-domain. If also $J(R)$ is principal, then $R$ is a B\'ezout domain.

In Section 5 we  show that for SP-domains with nonzero Jacobson radical, the groups $\Inv(R)$ and $\Div(R)$  are topological invariants of $\Max(R)$ (Theorems~\ref{characterization} and~\ref{EX cor}), and as a consequence we obtain that they are both free groups (Corollary~\ref{Bergman}).   This leads naturally to consideration of the group $\Div(R)/\Inv(R)$, and we show that it is torsion-free, but that it is divisible only when $\Div(R) =\Inv(R)$, with the latter condition characterized topologically also in terms of $\Max(R)$ (Theorem~\ref{torsion-free}).

In Section 6 we reframe Loper and Lucas' notion of the sharp and dull degrees of a one-dimensional Pr\"ufer domain in topological terms and use this observation to obtain existence results for such domains with prescribed sharp degrees.

\section{Preliminaries and Notation}

%\smallskip
%{\bf Define and discuss ${\rm Prin}(R)$, ${\rm Inv}(R)$, ${\rm Div}(R)$, and ${\rm Cl}(R)$. }
%\smallskip

Let $R$ be an integral domain with quotient field $F$.
We denote by
$J(R)$  the Jacobson radical of $R$.

\smallskip

{\bf (2.1)} {\it The group of invertible fractional ideals}.
 We denote by ${\rm Inv}(R)$ the group of invertible fractional ideals of $R$, and by
${\rm Prin}(R)$ the subgroup of ${\rm Inv}(R)$ consisting of the nonzero principal fractional ideals of $R$; i.e., ${\rm Prin}(R)$ is the {\it group of divisibility} of $R$.
For background on the group of divisibility of a domain, see \cite[Section 16]{Gil}.
  The  \textit{Picard group} of $R$ is  the quotient group ${\rm Pic}(R):={\rm Inv}(R)/{\rm Prin}(R)$.

\smallskip

{\bf (2.2)}  {\it The group of nonzero divisorial fractional ideals.}
The set of nonzero divisorial fractional ideals of $R$ is denoted
 ${\rm Div}(R)$. (Recall that a fractional ideal $I$ of $R$ is \textit{divisorial} if $(I^{-1})^{-1}=I$, where $I^{-1}=\{x\in F: xI\subseteq R\}$.)
  The domain $R$ is completely integrally closed if and only if $\mbox{Div}(R)$ is a group with the binary operation given by  $I\star J:=((IJ)^{-1})^{-1}$ for all $I,J \in {\rm Div}(R)$ \cite[Theorem 34.3]{Gil}.
  If $I,J$ are invertible fractional ideals of $R$, then $I \star J = IJ$.
   Thus
  both ${\rm Prin}(R)$ and ${\rm Inv}(R)$ are subgroups of ${\rm Div}(R)$. The quotient group ${\rm Cl}(R)={\rm Div}(R)/{\rm Prin}(R)$ is the \textit{class group} of $R$, and ${\rm Pic}(R)$ is a subgroup of ${\rm Cl}(R)$.
%{\bf Define and discuss $\ell$-groups.}

\smallskip

{\bf (2.3)} {\it Lattice-ordered groups}.
An abelian group $(G,+)$ is a   \textit{lattice ordered group ($\ell$-group)} if it is partially ordered with a relation $\leq$ such that the meet $x\wedge y$ and join $x\vee y$
of $x,y\in G$ exist and such that  for all $x,y,a\in G$, $x\leq y$ implies $a+x\leq a+y$. Throughout the paper, by an ``$\ell$-group'' we mean an abelian $\ell$-group.
%The $\ell$-group $G$ is \textit{complete} (see for example \cite{CA}) if each set of elements of $G$ that is bounded above by an element in the group has a least upper bound; equivalently, each set of elements that is bounded below has a greatest lower bound.
 An \textit{$\ell$-homomorphism} of $\ell$-groups $\phi:G\rightarrow G'$ is a group homomorphism  such that for all $x,y\in G$, $\phi (x\wedge y)=\phi(x)\wedge\phi(y)$ (and thus $\phi (x\vee y)=\phi(x)\vee\phi(y)$).

% {\bf (2.4)}  {\it The completion of an $\ell$-group.}
% The $\ell$-group $G$ is \textit{complete} if every subset of $G$ has a greatest lower bound and a least upper bound.
% An $\ell$-subgroup $G$ is \textit{dense} in an $\ell$-group $H$ if for $0<h\in H$, there is an $g\in G$ such that $0<g\leq h$. An $\ell$-group $G$ is \textit{archimedean} if whenever $a, b\in G$ with $na\leq b$ for all $n\in\mathbb{N}$, then $a\leq 0$.  The \textit{completion} of an $\ell$-group $G$ is the smallest complete $\ell$-group $H$ that contains $G$.  Each archimedean $\ell$-group has a completion {\bf [ref]}, and this completion  is characterized by the following lemma.

% \begin{lem}{\em (Conrad-McAlister \cite [Theorem 2.4]{CA})}\label{completion} Let $G$ be an archimedean $\ell$-subgroup of a complete $\ell$-group $H$. Then the following are equivalent.
%\begin{enumerate}
%\item[(1)] $H$ is the completion of $G$.
%\item[(2)] If $0<h\in H$, then $0<g_1\leq h\leq g_2$ for some $g_1, g_2\in G$.
%\item[(3)] $G$ is dense in $H$ and no proper $\ell$-subgroup of $H$ contains $G$ and is complete.
%\end{enumerate}
%\end{lem}
\smallskip

{\bf (2.4)}  {\it The $\ell$-group of  invertible fractional ideals of  a Pr\"ufer domain.}
The group ${\rm Inv}(R)$ admits a partial order given for all $I,J \in {\rm Inv}(R)$ by $I \leq J$ if and only if $I\supseteq J$. When $R$ is a Pr\"ufer domain, this partial order gives ${\rm Inv}(R)$ the structure of an $\ell$-group, with meet and join given, respectively, as $I \wedge J = I+J$ and  $I\vee J=I\cap J$ for all $I,J\in\mbox{Inv}(R)$ \cite[Theorem 2]{BK}. (In a  Pr\"{u}fer domain, finite intersections and finite sums of invertible ideals are invertible. In fact from \cite[Theorem 25.2(d)]{Gil}, we have $(I+J)(I\cap J)=IJ$. Now since $IJ$ is invertible, $I\cap J$ is also invertible.)  When also $R$ is a B\'ezout domain, ${\rm Inv}(R) = {\rm Prin}(R)$, and hence in this case ${\rm Prin}(R)$, the group of divisibility of $R$,  is an $\ell$-group.
 By the Krull-Kaplansky-Jaffard-Ohm Theorem \cite[Theorem 5.3, p.~113]{FS}, each  $\ell$-group is isomorphic to the group of divisibility of a B\'{e}zout domain.

%\smallskip

%{\bf (2.5)} {\it The complete $\ell$-group of nonzero divisorial fractional ideals of a completely integrally closed  domain.}
%When $R$ is a completely integrally closed domain, the group
% $\mbox{Div}(R)$ of nonzero divisorial fractional ideals of $R$  can be  partially ordered by $I \leq J$ if and only if $J \subseteq I$. With this partial order, Div$(R)$ admits the structure of a complete $\ell$-group,
% with  arbitrary join and meet given, respectively, by
% $\bigvee_{\alpha} I_\alpha =\bigcap_\alpha I_\alpha$ and
 % $\bigwedge I_{\alpha}= ((\sum_{\alpha} I_\alpha)^{-1})^{-1}$.  If also $R$ is a Pr\"ufer domain, then $\Inv(R)$ is an $\ell$-subgroup of $\Div(R)$.

\section{Completely integrally closed Pr\"ufer domains}

As discussed in (2.4), the group $\Inv(R)$ of  invertible fractional ideals of a Pr\"ufer domain $R$ is an $\ell$-group. We show in Proposition~\ref{completely integrally prufer} that  if also $R$ is a completely integrally closed domain, then $\Inv(R)$ is an archimedean $\ell$-group, and hence admits a completion that proves to be the group $\Div(R)$ of nonzero divisiorial fractional ideals of $R$.
We develop a ring-theoretic analogue of this  by showing that every completely integrally closed Pr\"ufer domain densely embeds in a pseudo-Dedekind B\'ezout domain. The {\it pseudo-Dedekind domains}, which play the role of the ``complete'' objects here, are the domains for which every nonzero divisorial ideal is invertible.
 (The terminology of ``pseudo-Dedekind,'' which is due to Anderson and Kang \cite{AK}, is motivated by Bourbaki's use of ``pseudo-principal'' to describe the domains for which every divisorial ideal is principal \cite[Exercise VII.1.21]{Bou}. Zafrullah \cite{Zaf} uses the term ``generalized Dedekind'' for what we term ``pseudo-Dedekind'' here.)
  In analogy with the fact that only archimedean $\ell$-groups have a completion, Proposition~\ref{completely integrally prufer} shows that  we must restrict to completely integrally closed Pr\"ufer domains.
  % We show in this same proposition that the completion of $\Inv(R)$  is the group  $\Div(R)$ of divisorial fractional ideals of $R$. %We then show $R$ sits ``densely'' in a faithfully flat extension $S$ such that $\Inv(S) \cong \Div(S)$. Finally, we replace $S$ with a B\'ezout domain with the same group of invertible fractional ideals.
We use throughout this section that a Pr\"ufer domain is completely integrally closed if and only if $\bigcap_{n>0} I^n = 0$ for all proper finitely generated ideals $I$ \cite[Theorem 8]{GH}.

%However, it is well-known that the converse is not true. Now suppose $R$ is a domain such that ${\rm Inv}(R)$ is a complete $\ell$-group. Let $I$ be a divisorial ideal of $R$. Then $I=\cap_{a\in F^{\ast}} aR$ with $I\subseteq aR$. Now for $0\neq x\in I$, we have $xR\subseteq aR$, i.e., $aR\leq xR$. So $\{a\}$ is a collection in ${\rm Inv}(R)$ that is bounded above and by the completeness of ${\rm Inv}(R)$, $J:=\bigvee aR\in {\rm Inv}(R)$. Now for each $y\in I$, $aR\leq yR$ and so $J=\bigvee aR\leq yR$. Thus $yR\subseteq J$ and $J=I$ is invertible. Hence if $R$ is a domain such that ${\rm Inv}(R)$ is a complete $\ell$-group, then $R$ is pseudo-Dedekind. The converse is also true.

%Let $R$ be a Pr\"ufer domain. Then as discussed in {\bf [ref]}, ${\rm Inv}(R)$ is  an $\ell$-group.

We recall first the notion of the (conditional) completion of an $\ell$-group; for additional details, see \cite[pp.~312--313]{Bir} and \cite{CA}.
An $\ell$-group $G$ is \textit{complete} if every set of elements of $G$ that is bounded below has a greatest lower bound (equivalently, every set of elements that is bounded above has a least upper bound).
 A complete $\ell$-group is necessarily {\it archimedean} \cite[Lemma 5, p.~291]{Bir}, meaning that
whenever $a, b\in G$ with $na\leq b$ for all $n >0$, then $a\leq 0$.

 An $\ell$-subgroup $G$ is \textit{dense} in an $\ell$-group $H$ if for $0<h\in H$, there is  $g\in G$ such that $0<g\leq h$.
  The \textit{completion} of an $\ell$-group $G$ is the smallest complete $\ell$-group $H$ that contains $G$.  Each archimedean $\ell$-group $G$ has a completion $H$, which is characterized by the  properties that (a) $H$ is a complete $\ell$-group containing $G$ as subgroup, and (b) for each $0 < h \in H$, there exist $g_1,g_2 \in G$ such that $0 < g_1 \leq h \leq g_2$ \cite[Theorem 2.4]{CA}.  In particular, $G$ is dense in its completion.

  % $G$ is dense in $H$ and no proper $\ell$-subgroup of $H$ contains $G$ and is complete \cite[Theorem 2.4]{CA}.

  In the next proposition we apply these ideas to the $\ell$-group $\Inv(R)$ for $R$  a Pr\"ufer domain. Note that
    ${\rm Inv}(R)$ is archimedean if whenever $I,J \in {\rm Inv}(R)$ with $J \subseteq I^n$ for all $n >0$, then $R \subseteq I$.

%We characterize in the next proposition when ${\rm Inv}(R)$ is archimedean, and show that when it is, ${\rm Div}(R)$ is the completion of $R$.

\begin{prop}\label{completely integrally prufer}  \label{completion of prufer} A Pr\"ufer domain $R$ is completely integrally closed if and only if ${\rm Inv}(R)$ is an archimedean $\ell$-group. If this is the case, the $\ell$-group
  ${\rm Div}(R)$ is the completion of the $\ell$-group ${\rm Inv}(R)$.
\end{prop}
\begin{proof}
Suppose that $R$ is completely integrally closed with $I,J \in \mbox{Inv}(R)$ such that $J \subseteq I^n$ for all $n >0$. Since $R$ is a Pr\"ufer domain, $I^n \cap R = (I \cap R)^n$ \cite[Theorem 2.2]{GG}, and hence $J \cap R \subseteq (I \cap R)^n$ for all $n>0$. Therefore, since $R$ is completely integrally closed and $J \cap R \ne 0$, it must be that $I \cap R = R$, and hence $R \subseteq I$.  Thus ${\rm Inv}(R)$ is an archimedean $\ell$-group.
Conversely, if $I$ is a proper finitely generated ideal of $R$ and ${\rm Inv}(R)$ is archimedean, then $\bigcap_{n>0} I^n = 0$, and hence $R$ is completely integrally closed.
%
%Conversely, let $x$ be a nonzero element of the quotient field of $R$ and $y$ an element of $R$. Suppose that $x^ny\in R$ for all $n\in\mathbb{N}$. Then $yR\geq (x^{-1}R)^n$ for all $n\in\mathbb{N}$, and so $x^{-1}R\leq R$, since $\mbox{Inv}(R)$ is archimedean. Thus $xR\geq R$, i.e., $x\in R$ and $R$ is completely integrally closed.
%\end{proof}

%\begin{prop}\label{completion of prufer}  If $R$ is a completely integrally closed Pr\"{u}fer domain, then the $\ell$-group
%  ${\rm Div}(R)$ is the completion of the $\ell$-group ${\rm Inv}(R)$ in the sense of Lemma~\ref{completion}.
%\end{prop}
%\begin{proof}

Now suppose $R$ is completely integrally closed.
Using \cite[Theorem 2.4]{CA} as discussed before the proposition,
 ${\rm Div}(R)$ is the completion of ${\rm Inv}(R)$ if and only if (a) ${\rm Div}(R)$ is complete, and (b)
 whenever $I$ is a proper divisorial ideal of $R$, there exist  invertible ideals $J$ and $K$  such that $K \subseteq I \subseteq J \subsetneq R$.
%We have established already that since $R$ is completely integrally closed, $\mbox{Inv}(R)$ is an archimedean $\ell$-group.
To see that ${\rm Div}(R)$ is complete, let $\{I_\alpha\}$ be a collection of ideals in ${\rm Div}(R)$ that is bounded below with respect to set inclusion. Then there exists a nonzero divisorial ideal $J$ contained in $I:=\bigcap_{\alpha}I_\alpha$. As an intersection of divisorial ideals, $I$ is divisorial, and since $J $ is nonzero, so is $I$. Therefore, $I \in {\rm Div}(R)$ and it follows that $I$ is the greatest lower bound of $\{I_\alpha\}$. This verifies (a).

Next, to verify (b),
let $I\in \mbox{Div}(R)$ with $I  \subsetneq R$, and write $I = \bigcap_{\alpha}q_\alpha R$, where the $q_\alpha$ are elements of the quotient field of $R$. Then there is $\alpha$ such that $I \subseteq R \cap q_{\alpha}R \subsetneq R$.  Since $R$ is a Pr\"ufer domain, the ideal $J:=R \cap q_\alpha R$ is invertible (cf.~\cite[Theorem 25.2(d)]{Gil}). Let $K$ be any nonzero principal ideal of $R$ contained in the nonzero ideal $I$. Then $K \subseteq I \subseteq J \subsetneq R$, which verifies (b).
%
%which proves that $\mbox{Inv}(R)$ is dense in $\mbox{Div}(R)$.
%
%  Finally, suppose that there is a complete $\ell$-group with $\mbox{Inv}(R)\subseteq H\subsetneq \mbox{Div}(R)$.
% Let $I \in   \mbox{Div}(R)\setminus H$, and write  $I = \bigcap_{\alpha} q_\alpha R$. Then for each $\alpha$, $q_\alpha R \in {\mbox{Inv}}(R)$, so that since $H$ is complete and hence closed under joins, $I=\bigcap_{\alpha} q_\alpha R \in H$, a contradiction. Therefore, $\mbox{Div}(R)$ is the completion of $\mbox{Inv}(R)$.
\end{proof}

 %If $R$ is Pr\"{u}fer, then the quotient group $\mbox{Div}(R)/\mbox{Inv}(R)$, denoted $Cl_w(R)$, is called the \textit{weak divisor class group} of $R$. Note that if $R$ is a B\'{e}zout domain, then $Cl(R)=Cl_w(R)$.

%\begin{cor}\label{completely-integrally} If $R$ and $S$ are completely integrally closed Pr\"{u}fer domains and there is an isomorphism of $\ell$-groups $\phi:{\rm{Inv}}(R) \rightarrow {\rm{Inv}}(S)$, then $\phi$ lifts to an isomorphism ${\rm Div}(R) \rightarrow {\rm Div}(S)$ of $\ell$-groups.
%\end{cor}
%\begin{proof} Suppose that $\mbox{Inv}(R)$ is ordered isomorphic to $\mbox{Inv}(S)$. By Proposition~\ref{completion of prufer}, $\mbox{Div}(R)$ is ordered isomorphic to $\mbox{Div}(S)$ as they are completions of the same group. Hence $Cl_w(R)$ is isomorphic to $Cl_w(S)$. {\bf Need that the isomorphism between ${\rm{Inv}}(R)$ and ${\rm{Inv}}(S)$ lifts to an isomorphism between  ${\rm{Div}}(R)$ and ${\rm{Div}}(S)$.}
%\end{proof}

\begin{defi}
We say that for an extension $R \subseteq S$ of Pr\"ufer domains, $R$  is {\it dense} in $S$ if the extension $R \subseteq S$ is faithfully flat and  for  each finitely generated proper  ideal $J$ of $S$ there is a finitely generated proper ideal $I$ of $R$ such that $J \subseteq IS$.
\end{defi}

\begin{lem} \label{m dense}  An extension $R \subseteq S$  of  Pr\"ufer domains is dense  if and only if  the $\ell$-homomorphism $\gamma:{\rm Inv}(R) \rightarrow {\rm Inv}(S):I \mapsto IS$ is a dense embedding of $\ell$-groups.
\end{lem}

\begin{proof}
Suppose that $R$ is dense in $S$.  To see that $\gamma$ is injective, let $I$ and $J$ be invertible fractional ideals of $R$ such that $IS = JS$.  Then $IJ^{-1}S = S$. Since $I$ and $J^{-1}$ are fractional ideals of $R$, there exists $0 \ne r \in R$ such that $rIJ^{-1}$ is an ideal of $R$.
Since $S$ is  flat, $S$ commutes with finite intersections of ideals of $R$ \cite[Theorem 7.4(ii), p.~48]{Mat}, so that
 $$(rR \cap rIJ^{-1})S = rS \cap rIJ^{-1}S = rS.$$  Thus, since $r$ is a nonzerodivisor in $S$, $(R \cap IJ^{-1}) S = S$. Since $S$ is faithfully flat,  every proper ideal of $R$ survives in $S$ \cite[Theorem~7.2(iii), p.~47]{Mat}, which implies that  $R \cap IJ^{-1} = R$.  Thus $R \subseteq IJ^{-1}$ and, since $J$ is invertible,  $J \subseteq I$.  Similarly, from $IS = JS$ we deduce that $S =I^{-1}J S$ and the same argument with the roles of $I$ and $J$ reversed shows that $I \subseteq J$.  Thus $I = J$ and
 $\gamma:{\rm Inv}(R) \rightarrow {\rm Inv}(S)$ is injective.

 To see that $\gamma$ is a dense embedding of $\ell$-groups, observe that if  $J$ is a proper invertible ideal of $S$, then by assumption there exists a proper finitely generated ideal $I$ of $R$ with $J \subseteq IS \subsetneq S$. Thus  $\gamma$ is a dense embedding. Moreover, $\gamma$ is an $\ell$-group homomorphism since for all invertible fractional ideals $I$ and $J$ of $R$, $(I+J)S = IS + JS$ and (by the flatness of $S$) $(I \cap J)S = IS \cap JS$   \cite[Theorem 7.4, p.~48]{Mat}.

 Conversely, suppose that $\gamma$ is a dense embedding.
Since $\gamma$ is injective, every proper invertible ideal $I$ of $R$ survives in $S$. If $I$ is a not necessarily invertible proper ideal of $S$ such that $IS = S$, then there is a finitely generated, hence invertible, proper ideal $J$ of $R$ with $J \subseteq I$ and $JS = S$, a contradiction. Thus every proper ideal of $R$ survives in $S$. Since every torsion-free module over a Pr\"ufer domain is flat \cite[Theorem 9.10, p.~233]{FS}, it follows that $S$ is a faithfully flat extension of $R$ \cite[Theorem 7.2, p.~47]{Mat}. Moreover,
 if $J$ is a proper finitely generated ideal of $S$, then, since $\gamma$ is a dense embedding, there exists a proper finitely generated  ideal $I$ of $R$   such that $J \subseteq \gamma(I) = IS \subsetneq S$.
\end{proof}

%In the next remark we note several alternative ways to characterize the conditions in Lemma~\ref{m dense}.

\begin{rem} \label{ff} {\em
Let  $R \subseteq S$ be an extension  of domains with $R$ a Pr\"ufer domain.
 The extension $R \subseteq S$ is faithfully flat if and only if $R = S \cap Q(R)$, where $Q(R)$ is the quotient field of $R$.
Indeed,  every torsion-free module over a Pr\"ufer domain  is flat \cite[Theorem 9.10, p.~233]{FS},
so it suffices to observe that $R = S \cap Q(R)$ if and only if  every maximal ideal of $R$ survives in $S$. If $R = S \cap Q(R)$ and $M$ is a maximal ideal of $R$, then since $M$ is a flat $R$-module,   $M = MS \cap MQ(R) = MS \cap Q(R)$ \cite[Theorem 7.4(i), p.~48]{Mat},
 and hence $S \not = MS$.    Conversely, if a maximal ideal of $R$ survives in $S$, then it
survives in $S \cap Q(R)$, so $S \cap Q(R)$ is
a faithfully flat ring between $R$ and $Q(R)$, which forces
 $R = S \cap Q(R)$ \cite[Exercise 7.2, p.~53]{Mat}}.
%(2)
 % It straightforward to verify that $R$ is dense in $S$
%if and only if  $S$ is a faithful $R$-module for which
% every proper finitely generated ideal of $S$ is contained in an ideal of the form $MS$, where $M$ is a maximal ideal of $R$.
%The claim of ``only if'' is clear. To see the converse,
% suppose that every proper finitely generated ideal of $S$ is contained in an ideal of the form $MS$, where $M$ is a maximal ideal of $R$.   Let $J$ be a proper finitely generated ideal of $S$, say
% $J = (x_1,\ldots, x_n)S$ for $x_1,\ldots,x_n \in J$.
 %  Then by assumption $J \subseteq MS$ for some maximal ideal $M$ of $S$, so  for each $k =1,\ldots,n$,  since $x_k \in MS$, there exists a finitely generated ideal $I_k$ of $R$ such that $x_k \in I_k \subseteq M$.  Thus $J \subseteq (I_1+\cdots+I_k)S \subseteq MS \subsetneq S$, which proves the claim. %  Since $S$ is a Pr\"ufer domain, $I:=I_1 + \cdots + I_n$ is an invertible ideal of $S$.  Thus $J \subseteq IS \subsetneq S$, and this proves that $R$ is $m$-dense in $S$.
%}
\end{rem}

%{\bf Discuss Nagata function ring of a Pr\"ufer domain.}
%Let $R$ be an integrally closed domain with quotient field $L$. The \textit{Kronecker function ring} of $R$ denoted $R^b$ is the ring
%$$R^b=\{\frac{f}{g}:f, g\in R[X], g\neq 0, c(f)^b\subseteq c(g)^b\},$$
%where $c(h)$ denotes the fractional ideal generated by coefficient of $h$ and $c(h)^b$ is the integral closure of $c(h)$ in $K$.
%It is well known that $R^b$ is a B\'{e}zout domain with quotient field $L(X)$ and $R^b\cap L=R$.

Let $R$ be a domain, and let $T$ be an indeterminate for $R$. The {\it Nagata function ring} $R(T)$ of  $R$ is given by
 $$R(T):=\left\{\frac{f}{g}:f, g\in R[T] \mbox{ and } c(g)=R\right\},$$  where $c(g)$ denotes the content ideal of $g$ in $R$.   The ring    $R(T)$ is a B\'{e}zout domain if and only if $R$ is a Pr\"{u}fer domain \cite[Theorem 33.4]{Gil}.  In Theorem~\ref{dense subring}  we apply   the Nagata function ring construction  to embed a Pr\"ufer domain inside a B\'ezout domain with the same group of invertible fractional ideals. This is done via the following lemma, which can be deduced from more general results (see for example  \cite[Proposition 3.1.23, p.~144]{Moc} and \cite[p.~589]{Ohm}). For lack of an explicit reference, we give a proof here.

% We show in the next lemma that when $R$ is a Pr\"ufer domain, the embedding $R \rightarrow R(T)$ is not only dense, but preserves the group of invertible fractional ideals of $R$.

\begin{lem} \label{invertible} If $R$ is a Pr\"ufer domain, then
the mapping
$$\phi : {\rm Inv}(R)\rightarrow {\rm Prin}(R(T)):I\mapsto IR(T)$$
is an isomorphism of $\ell$-groups.

\end{lem}
\begin{proof}  Since $R(T)$ is a B\'ezout domain, $\phi$ is well-defined.
 The extension $R \subseteq R(T)$ is   faithfully flat \cite[p.~39]{FO}, so $\phi$ is injective and preserves meets and joins \cite[Theorem 9.9 and Lemma 9.11]{FS}. Thus $\phi$ is also an $\ell$-group homomorphism.
 To show that $\phi$ is onto, let $S = R(T)$, and  let $J\in\mbox{Prin}(S)$. Then $J=fg^{-1}S$ for some $f,g \in R[T]$ with $c(g) = R$. Now $fS=c(f)S$
 %and $gS=c(g)S = S$
 \cite[Theorem 32.7]{Gil} and $g$ is a unit in $S$, so
 %Since $c(g)$ is a nonzero finitely generated ideal of the Pr\"ufer domain $R$, there is a fractional ideal $A$ of $R$ such that $c(g)A=R$.
   $c(f)S=fg^{-1}S=J$. Thus $\phi$ is onto.
%
% It remains to show that $\phi$ is an $\ell$-group homomorphism. It is clear that $\phi$ is a group homomorphism, and it is also obvious that $\phi$ is order-preserving: if $I\leq J$, then $I\supseteq J$, so that $IS \supseteq JS$, and hence $\phi(I) \leq \phi(J)$. In addition, for each $I, J\in\mbox{Inv}(R)$, $\phi(I\wedge J) =\phi(I+J)=(I+J)S=IS+JS=\phi(I)\wedge\phi(J)$. Now note that if $K\in \mbox{Prin}(S)$, $K=IS$ for some $I\in\mbox{Inv}(R)$ since $\phi$ is onto. So $K\cap R=IS\cap R=I$ for some $I\in\mbox{Inv}(R)$. Thus $\phi(I\vee J) =IS\cap JS=AS$ for some $A\in\mbox{Inv}(R)$. It then follows that $A=AS\cap R=(IS\cap R)\cap (JS\cap R)=I\cap J$. So $\phi(I\vee J) =AS=(I\cap J)S=\phi(I\vee J)$. Hence $\phi$ is an isomorphism of  $\ell$-groups.
 \end{proof}

The next lemma is an application of an existence theorem due to
Rump and Yang \cite[Corollary 3]{RY} that shows that if $R$ is a B\'ezout domain and  $f:{\rm Inv}(R) \rightarrow H$ is an embedding of $\ell$-groups, then there exists an extension $S$ of $R$ such that $S$ is a B\'ezout domain with  ${\rm Prin}(S) \cong H$ as $\ell$-groups. The proof of the lemma uses the notion of a B\'ezout valuation from \cite{RY}. Let $K$ be a field, and let $G$ be an $\ell$-group. We adjoin an element $\infty$ to $G$ such that  $\infty > g$ and $\infty + g = \infty$ for all $g \in G$.
 A map $\nu:K \rightarrow G \cup \{\infty\}$ is a {\it B\'ezout valuation} if for all $x,y \in K$,
\begin{itemize}
\item[(a)] $\nu(0) = \infty$, $\nu(1) = 0$,
\item[(b)] $\nu(xy) = \nu(x) + \nu(y)$,
\item[(c)]  $\nu(x) \wedge \nu(y) \leq \nu(x+y)$, and
\item[(d)] $\nu(x) \wedge \nu(y)  = \nu(xz+yw)$ for some  $z,w \in K$ with $\nu(z),\nu(w)\geq 0$.
 \end{itemize}

\begin{lem} \label{new commutes} Let $R$ be a Pr\"ufer domain, and let $ f:{\rm Inv}(R) \rightarrow H$ be an embedding of $\ell$-groups. Then there is a ring extension $S$ of $R$ such that $S$ is a
 B\'ezout domain and there is an  isomorphism of $\ell$-groups $\psi:{\rm Inv}(S) \rightarrow H$ for which
 % $\psi^{-1} \circ f:
%  {\rm Inv}(R) \rightarrow {\rm Inv}(S)$  is  the   $\ell$-group homomorphism given by
$\psi^{-1}(f(I)) = IS$ for all $I \in {\rm Inv}(R)$.
 \end{lem}

 \begin{proof} Let $A = R(T)$, so that $A$ is a B\'ezout domain, and let $\phi:{\rm Inv}(R) \rightarrow {\rm Inv}(A):I \rightarrow IA$ be the $\ell$-group isomorphism given in Lemma~\ref{invertible}.
 Let $K$ be the quotient field of $A$. Then there is a surjective B\'ezout valuation $\nu_A:K \rightarrow {\rm Inv}(A) \cup \{\infty\}$ such that $A = \{x \in K:\nu_A(x) \geq 0\}$ \cite[Proposition 1]{RY}.
 Denote by $g$ the map from $\Inv(A) \cup \{\infty\}$ to $H \cup \{\infty\}$ such that $g|_{\Inv(A)} = f \circ \phi^{-1}$ and $g(\infty) = \infty$.
 By \cite[Corollary 3]{RY} there is an extension field $L$ of $K$ and a surjective B\'ezout valuation $\nu:L \rightarrow H \cup \{\infty\}$ such that the following diagram commutes.
 \begin{center}
$\begin{CD}
K  @>{\nu_A}>> \Inv(A) \cup \{\infty\} \\
@VV{\subseteq}V        @VV{g}V\\
L    @>{\nu}>>  H \cup \{\infty\}.
\end{CD}$
 \end{center}
Let $S = \{x \in L:\nu(x) \geq 0\}$. Then $S$ is a B\'ezout domain with quotient field $L$ \cite[Proposition 1]{RY}.  Define $\psi:{\rm Inv}(S) \rightarrow H$ by $\psi(xS) = \nu(x)$ for each $0\ne x \in L$.   That $\psi$ is well defined and injective follows from the fact that for $x,y \in L$, $xS = yS$ if and only if  $\nu(x) = \nu(y)$. Since $\nu$ is surjective, so is $\psi$. Also, from the fact that $\nu(xy) = \nu(x)+\nu(y)$ for all $0 \ne x,y \in L$, it follows that $\psi$ is a group homomorphism.

To see that $\psi$ is an $\ell$-group homomorphism, let $0 \ne x,y \in L$.  %Since $S$ is a B\'ezout domain, there exists $z \in L$ such that $xS + yS = zS$.  Also,
Since $\nu$ is a B\'ezout valuation, there exists $z \in xS + yS$ such that $\nu(x) \wedge \nu(y) = \nu(z)$, and hence $xS + yS = zS$ (see the proof of  \cite[Proposition 1]{RY}). Therefore,  $$\psi(xS \wedge yS) = \psi(xS +yS) = \psi(zS) = \nu(z) = \nu(x) \wedge \nu(y)= \psi(xS) \wedge \psi(yS),$$
proving that $\psi$ is an $\ell$-group isomorphism.

Finally, let $I \in {\rm Inv}(R)$. Since $A$ is a B\'ezout domain, there exists  $x \in K$ such that $IA = xA$.    Since $\phi$ is an isomorphism, $I = \phi^{-1}(xA)$, and hence  $$\psi^{-1}(f(I)) = \psi^{-1}(f(\phi^{-1}(xA))) = \psi^{-1}(\nu(x)) = xS = IS.$$ Therefore, $\psi^{-1} \circ f$ is the canonical map from ${\rm Inv}(R)$ to ${\rm Inv}(S)$.
 \end{proof}

%\begin{lem} \label{RY cor} If $R$ is a Pr\"ufer domain, $G$ is an $\ell$-group and $f:{\rm Inv}(R) \rightarrow G$ is an embedding of $\ell$-groups, then there exists an extension $S$ of $R$ such that $S$ is a B\'ezout domain with ${\rm Inv}(S) = G$.
%\end{lem}

%\begin{proof}    Since $R$ is a Pr\"ufer domain,  Lemma \ref{invertible} implies that ${\rm Inv}(R)\cong{\rm Prin}(R(T))$ as $\ell$-groups. By the result of Rump and Yang cited above, there exists an extension $S$ of $R(T)$  such that $S$ is a B\'ezout domain and  ${\rm Prin}(S) \cong G$. Since $S$ is a B\'ezout domain, ${\rm Prin}(S) = {\rm Inv}(S)$, so the lemma is proved.
%\end{proof}

We apply the lemma in the case in which $R$ is completely integrally closed and $H = \Div(R)$. We also make use of the fact that
%Every pseudo-Dedekind domain is completely integrally closed \cite[p.~327]{AK}.
 a Pr\"ufer domain $R$ is pseudo-Dedekind if and only if ${\rm Inv}(R)$ is a complete $\ell$-group \cite[p.~327]{AK}.

\begin{thm}\label{dense subring} A Pr\"ufer domain $R$ is  completely integrally closed if and only if $R$  is a dense subring of
 a pseudo-Dedekind B\'ezout domain $S$. Moreover, $S$ can be chosen such that
  ${\rm Inv}(S) \cong {\rm Div}(R)$ as $\ell$-groups.

\end{thm}

\begin{proof}  Suppose that $R$ is  completely integrally closed. By Proposition~\ref{completely integrally prufer},
the inclusion mapping
 $f:{\rm Inv}(R) \rightarrow {\rm Div}(R)$ is a dense embedding of $\ell$-groups.    Since $R$ is a Pr\"ufer domain,  Lemma~\ref{new commutes} implies there exists a B\'ezout domain $S$   having $R$ as a subring  such that there is an  isomorphism of $\ell$-groups $\psi:{\rm Inv}(S) \rightarrow \Div(R)$ with $\gamma:=\psi^{-1} \circ f$  the canonical  $\ell$-group homomorphism that sends
$ I \in  {\rm Inv}(R)$ to $IS \in {\rm Inv}(S)$.
Since, by Proposition~\ref{completely integrally prufer},  ${\rm Div}(R)$ is complete, so is ${\rm Inv}(S)$.
 As a complete $\ell$-group,  ${\rm Inv}(S)$ is archimedean. Thus,
 by Proposition~\ref{completely integrally prufer},  $S$ is  a completely integrally closed B\'ezout domain with   $\Inv(S) = \Div(S)$. Therefore,  $S$ is a pseudo-Dedekind domain.

To show that $R$ is dense in $S$, it suffices by
 Proposition \ref{completion of prufer} to show that
$\gamma$ is a dense embedding. Since $\gamma$ is a composition of injective maps, $\gamma $ is injective. To see that $\gamma$ is dense, let $I \in {\rm Inv}(S)$ such that $I \subsetneq S$.  Then, since $\psi$ is an isomorphism of $\ell$-groups, $\psi(I) \subsetneq R$.  Since the inclusion mapping $f$ is a dense embedding, there exists $J \in {\rm Inv}(R)$ such that $\psi(I) \subseteq f(J) \subsetneq R$.  Therefore, $I \subseteq \psi^{-1}(f(J)) \subsetneq S$, and since $\gamma(J) = \psi^{-1}(f(J))$, this proves $\gamma$ is a dense embedding.

Conversely, if $R$ is dense in a pseudo-Dedekind B\'ezout domain $S$ and $I$ is a proper finitely generated ideal of $R$, then since $I$ survives in the completely integrally closed B\'ezout domain $S$, we have $\bigcap_{n>0}I^n \subseteq \bigcap_{n>0}I^nS = 0$. Hence $R$ is  completely integrally closed.
\end{proof}

In Section 5 we apply Theorem~\ref{dense subring} to SP-domains with nonzero Jacobson radical. In that case the groups $\Inv(R)$ and $\Div(R)$ are topological invariants of the maximal spectrum of $R$. Thus
 the embedding in Theorem~\ref{dense subring} ultimately depends  on $\Max(R)$.

\section{SP-domains}

Following Vaughan and Yeagy \cite{VY}, we say that a ring $R$ for which every proper
ideal is a product of radical ideals (i.e., ``semi-prime'' ideals) is
an {\it SP-ring}.  These rings have been studied by several authors; see for example \cite{BY, FHL, Has, WM, Olb, RE, VY, Y}.
 An SP-domain is  necessarily an
  {\it almost Dedekind domain}; that is, $R_M$ is a
Dedekind domain for each maximal ideal $M$ of $R$ \cite[Theorem 2.4]{VY}. In particular, an SP-domain is a one-dimensional Pr\"ufer domain.  Of course there exist SP-domains that are not B\'ezout domains, since Dedekind domains are SP-domains. However, there are more subtle examples also.

It is possible to construct a variety of SP-domains by using integral extensions of Dedekind domains. More precisely, one can obtain examples of SP-domains that are neither Dedekind domains nor B\'ezout domains by using the following method (which was already known to Grams \cite{GR}).

\begin{ex} {\em Let $R$ be a Dedekind domain with quotient field $K$ such that ${\rm Max}(R)$ is countable. Let $\{K_i\}_{i\geq 0}$ be an ascending chain of extension fields of $K$ such that $K_0=K$, and $[K_i:K]<\infty$ for all $i\geq 0$. For $i\geq 0$ let $R_i$ be the integral closure of $R$ in $K_i$. For each $i\geq 0$, $R_i$ is a Dedekind domain with countably many maximal ideals. Set $S=\bigcup_{i\geq 0} R_i$ (which is the integral closure of $R$ in $\bigcup_{i\geq 0} K_i$). Suppose that $\{M_i\}_{i\geq 0}$ is the sequence of all maximal ideals of $R$.

\begin{itemize}
\item [(1)] If for every $i>0$ and every maximal ideal $Q$ of $R_i$ lying over some $M_j$ with $j\leq i$, it follows that $QR_{i+1}$ is a square-free product of maximal ideals of $R_{i+1}$, then $S$ is an SP-domain (see \cite[Proposition 6.5.2]{REK}).
\item [(2)] If there is an ascending chain $\{Q_j\}_{j\geq 0}$, where $Q_{i+1}$ is a maximal ideal of $R_{i+1}$ lying over the maximal ideal $Q_i$ of $R_i$ and $Q_{i+1}\ne Q_iR_{i+1}$ for each $i\geq 0$, then $S$ is not a Dedekind domain (see \cite[Proposition 6.5.5]{REK}).
\item [(3)] Suppose that $M_0 \subseteq \bigcup_{i>0} M_i$. If for every $i>0$, $QR_{i+1}$ is a maximal ideal of $R_{i+1}$ whenever $Q$ is a maximal ideal of $R_i$ lying over some $M_j$ with $0<j\leq i$, then $S$ is a domain that satisfies the ascending chain condition for principal ideals (ACCP) (see \cite[Proposition 6.5.4]{REK}).
\end{itemize}
If all the residue fields of $R$ are finite, $M_0$ is not principal and the Picard group of $R$ is torsion-free, then it is possible to construct an ascending chain $\{K_i\}_{i\geq 0}$ of extension fields of $K$ such that (1), (2), and (3) are satisfied and $[K_{i+1}:K_i]=2$ for each $i\geq 0$ (see \cite[Theorem 42.5]{Gil}). (Observe that $M_0 \subseteq \bigcup_{i>0} M_i$ since ${\rm Pic}(R)$ is torsion-free and $M_0$ is not principal.)

To give an explicit example, we consider the domain $\overline{D}$ in \cite[Example 1]{GR}, which can be obtained by using the previous construction. By the prior remarks it is evident that $\overline{D}$ is an SP-domain that satisfies  ACCP and that is not a Dedekind domain. In particular, $\overline{D}$ is not a B\'ezout domain, because a B\'ezout domain that satisfies ACCP is a PID. (If a B\'ezout domain satisfies  ACCP, then it satisfies  ACC for finitely generated ideals, which is equivalent to being Noetherian, hence it is a PID.)}
\end{ex}
SP-domains are characterized in several ways in \cite[Theorem 2.1]{Olb}. We strengthen some of the characterizations in the following lemma. Recall that a domain is {\it treed} if any two prime ideals contained in a maximal ideal of $R$ are comparable with respect to set inclusion.

\begin{lem}  \label{SP}
The following  are equivalent for a
 domain $R$ that is not a field.

\begin{itemize}

\item[(1)]  $R$ is an {\it SP}-domain.

\item[(2)]  $R$ is an almost Dedekind domain for which each maximal ideal of $R$ contains a finitely generated ideal that is not contained in the square of any maximal ideal of $R$.

\item[(3)]  Each proper nonzero ideal $A$ of $R$ can be represented uniquely as
a  product $A = J_1J_2\cdots J_n$ of radical ideals $J_i$ such that $J_1
\subseteq J_2 \subseteq \cdots \subseteq J_n$.

\item[(4)]  $R$ is a treed domain for which the radical of each  proper nonzero principal ideal  is  invertible.

%\item[(5)]
%$R$ is a one-dimensional domain for which the radical of every nonzero proper principal ideal  is invertible.

%\item[(iv)]  For every proper ideal $A$ of $R$, there exist radical ideals $J_1
%\subseteq J_2 \subseteq \cdots \subseteq J_n$ such that $A = J_1J_2 \cdots
%J_n$.

\end{itemize}
\end{lem}

\begin{proof} (1) $\Rightarrow$ (2) An SP-domain is necessarily an
  {almost Dedekind domain} \cite[Theorem 2.4]{VY}. Thus (2) follows from \cite[Theorem 2.1]{Olb}.

%(2) $\Rightarrow$ (3) A one-dimensional domain is a treed domain, so (3) is a consequence of \cite[Theorem 2.1]{Olb}.

(2) $\Rightarrow$ (3) This is proved in \cite[Theorem 2.1]{Olb}.

(3) $\Rightarrow$ (4) It is clear that (3) implies $R$ is an SP-domain. Since an SP-domain that is not a field is an almost Dedekind domain,  the only nonmaximal prime ideal is $(0)$, so that $R$ is treed. By \cite[Theorem 2.1]{Olb}, an SP-domain has the property that the radical of every nonzero finitely generated ideal is invertible.

(4) $\Rightarrow$ (1)  We first show that $R$ is an almost Dedekind domain.
Let $M \in \Max(R),$ and suppose that $P$ and $Q$ are prime ideals with $P \subsetneq Q \subseteq M$.
%It is sufficient to show that $R_M$ is one-dimensional. Let $P \subseteq Q$ be nonzero prime ideals, and suppose that $P \ne Q$.
Let $x \in P$ be nonzero, and let $y \in Q \setminus P$.  Then  $\sqrt{xR} \subsetneq \sqrt{yR}$ are by assumption invertible ideals. Thus, since $R_M$ is a treed domain,  $(\sqrt{xR})R_M \subseteq (\sqrt{yR})R_M$ is a containment of principal prime ideals. Since $R_M$ is a domain, this forces,  $(\sqrt{xR})R_M  = (\sqrt{yR})R_M$. But then $y \in (\sqrt{xR})R_M \subseteq PR_M$,  contrary to the assumption that $y \not \in P$. This shows that $R_M$ is a one-dimensional domain. Since also the radical of every nonzero proper principal ideal is invertible, it follows that $R_M$ is a DVR.  Therefore, $R$ is an almost Dedekind domain.  Finally, from (4) we have that every maximal ideal of $R$ contains an invertible radical ideal, and hence by  \cite[Proposition 1.3]{WM}, $R$ is an SP-domain.
\end{proof}

Using the lemma we show next that whether an SP-domain is B\'ezout depends only on the radicals of principal ideals.

%\begin{rem} {\em The first part of the  proof can be modified to show that
%if $R$ is a one-dimensional domain with $J(R) \ne 0$, then for each finitely generated  ideal $I$ of $R$, there exists $x \in I$ such that  $\sqrt{I} = \sqrt{xR}$. Thus a one-dimensional Pr\"ufer domain $R$ with $J(R) \ne 0$ is a QR-domain, meaning that every overring of $R$ is a localization  of $R$ {\bf [ref]}. }
%\end{rem}

%We improve on the characterization in Lemma~\ref{SP}(3) by showing that in order for a one-dimensional domain $R$ to be an SP-domain, it suffices that the radical of every principal ideal is finitely generated.

\begin{thm} \label{BSPG}  \label{SPG}
A domain $R$ is
 a B\'ezout SP-domain if and only if  $R$ is a treed domain such that the radical of every  principal ideal of $R$ is  principal.
\end{thm}

\begin{proof}
We may assume $R$ is not a field since otherwise the claim is clear.
If $R$ is a B\'ezout SP-domain, then by Lemma~\ref{SP}  the radical of each  principal ideal of $R$ is  principal.
Conversely, suppose that $R$ is a treed domain for which the radical of every  principal ideal is principal. By Lemma~\ref{SP}, $R$ is an SP-domain, so it remains to show that $R$ is a B\'ezout domain. We do this by proving two claims.
 \smallskip

{\textsc{Claim 1:}} If $a,b,c$ are nonzero in $R$ with $aR,bR$ radical ideals and $a = bc$, then
\[
\{M\in\Max(R): a\in bM\}=\{M\in\Max(R): a\in M,b\not\in M\}.
\]

To verify the claim, let $M\in\Max(R)$. Suppose $a\in bM$. Since $a = bc$, we have $c \in MR_M$.
 Since $R_M$ is a DVR and $aR_M$ is a nonzero radical ideal, it must be that $MR_M = aR_M = bcR_M$. Therefore, since $c \in MR_M$, it cannot be that $b \in M$.
 The reverse inclusion is clear in light of the fact that $a = bc$.
 \smallskip

{\textsc{Claim 2:}} $\sqrt{xR+yR}$ is principal for all $x,y\in R$.

\smallskip

To prove Claim 2,  let $x,y\in R$. We may assume $x$ and $y$ are nonzero. By assumption, there exist $a,b,c\in R$ such that $\sqrt{xyR}=aR$, $\sqrt{xR}=bR$, and $\sqrt{yR}=cR$. Note that $aR = bR \cap cR$. Moreover, there is  $d\in R$ such that $\sqrt{(a/b)R}\cap\sqrt{({a}/{c})R}=dR.$ %Observe that $aR \subseteq dR$.
By Claim 1 we have
\begin{eqnarray*} \{M\in\Max(R): d\in M\} & = &\{M\in\Max(R):a\in bM\}\cup\{M\in\Max(R):{a}\in cM\} \\
\: & = & \{M\in\Max(R): a\in M,(b\not\in M\textnormal{ or }c\not\in M)\}.
\end{eqnarray*}
Therefore, Claim 1 implies that
\begin{eqnarray*} \{M\in\Max(R): x,y\in M\} &= & \{M\in\Max(R): b,c\in M\} \\
& = & \{M\in\Max(R): a\in M,d\not\in M\} \\
& = & \{M\in\Max(R):{a}\in dM\}.
\end{eqnarray*} Consequently, $\sqrt{xR+yR}$ is the radical of the principal ideal $({a}/{d})R$, and hence by assumption  principal.

Using Claim 2 it follows by induction that every finitely generated radical ideal of $R$ is principal. Since $R$ is an SP-domain, every nonzero finitely generated ideal of $R$ is a product of  finitely generated, hence principal, radical ideals. Therefore, every finitely generated ideal of $R$ is principal and  $R$ is a B\'ezout domain.
\end{proof}

\begin{cor} \label{BSPG2}  \label{SPG2}
A domain $R$ is a pseudo-Dedekind SP-domain if and only if $R$ is a treed domain such that the radical of every nonzero divisorial ideal is invertible. If also the radical of every nonzero divisorial ideal is principal, then $R$ is a B\'ezout domain.
\end{cor}

\begin{proof}  We may assume $R$ is not a field. If $R$ is a pseudo-Dedekind SP-domain, then since every nonzero divisorial ideal is invertible, Lemma~\ref{SP} implies that the radical of every nonzero divisorial ideal is invertible.
Conversely, suppose $R$ is a treed domain for which the radical of every nonzero divisorial ideal of $R$ is invertible. By Lemma~\ref{SP},
 $R$ is an SP-domain. Let $I$ be a nonzero divisorial ideal of $R$. By Lemma~\ref{SP}(3), $I = \sqrt{I}A$ for some ideal $A$ such that $A = R$ or $A$   is a product of radical ideals of $R$.  Since $I$ is divisorial and by assumption $\sqrt{I}$ is invertible, it follows that $A$ is divisorial. An inductive argument (which by Lemma~\ref{SP}(3) terminates after  finitely many steps) now shows that $I$ is a product of invertible ideals, and hence is itself invertible. Therefore, $R$ is a pseudo-Dedekind domain. The last statement follows from Theorem~\ref{SPG}.
\end{proof}

\begin{rem} {\em As noted in the proof of Theorem~\ref{BSPG}, a one-dimensional domain $R$ is an SP-domain if and only if each maximal ideal of $R$ contains an invertible radical ideal (see \cite[Proposition 1.3]{WM}). However, it is not true that $R$ is a B\'ezout SP-domain if and only if each maximal ideal of $R$ contains a nonzero principal radical ideal.
There exists a Dedekind domain $R$ with no principal maximal ideals and a sequence $\{M_i\}_{i=1}^\infty$ of all maximal ideals of $R$ such that $M_iM_{i+1}$ is principal for each $i$ (see \cite[Example 3-2]{CL}). In particular, $R$ is an SP-domain and every nonzero radical ideal of $R$ contains a nonzero principal radical ideal, yet $R$ is not a B\'ezout domain.}
\end{rem}

\begin{rem} {\em Let $R$ be an integral domain. Recall that $R$ is a factorial domain (or unique factorization domain) if every principal ideal of $R$ is a finite product of principal prime ideals. Furthermore, $R$ is called radical factorial (see \cite{RE}) if each of its principal ideals is a finite product of principal radical ideals. These notions are the ``principal ideal analogues'' of Dedekind domains and SP-domains. We  discuss how these classes of domains fit into the theory presented in this section. For instance, it is known that if the radical of every principal ideal of $R$ is principal, then $R$ is radical factorial (see \cite[Proposition 2.10]{RE}). By using the same methods as in the proof of \cite[Proposition 3.11]{RE} it can be shown that every treed radical factorial domain is an SP-domain. But even a Dedekind domain can fail to be radical factorial (see \cite[Example 4.3]{RE}). Clearly, every SP-domain with trivial Picard group is radical factorial (see \cite[Proposition 3.10(2)]{RE}). Note that $R$ is a Dedekind domain if and only if $R$ is an SP-domain of finite character (i.e., every nonzero element of $R$ is contained in only finitely many maximal ideals). It can be shown that $R$ is factorial if and only if $R$ is of ``finite height-one character'' (i.e., every nonzero element is contained in only finitely many height-one prime ideals) and the radical of every principal ideal of $R$ is principal (see \cite[Theorem 2.14]{RE}). On the other hand a radical factorial domain of finite height-one character (and finite character) can fail to be factorial (see \cite[Example 4.3]{RE}). It is well known that $R$ is a PID if and only if $R$ is a factorial Pr\"ufer domain.

(Since a factorial domain has trivial class group and thus trivial Picard group, it follows that a factorial Pr\"ufer domain is a B\'ezout domain. Moreover, a factorial domain satisfies ACCP. Therefore, a factorial Pr\"ufer domain is a B\'ezout domain that satisfies ACCP, and thus it is a PID.)

This shows that a factorial domain that is not a PID (e.g. $\mathbb{Q}[X,Y]$) is an example of an integral domain where the radical of every principal ideal is principal, but that is neither an SP-domain nor a B\'ezout domain. In particular, we infer that the property ``treed'' cannot be omitted in Lemma~\ref{SP} and Theorem~\ref{BSPG}.}
\end{rem}

We characterize next the SP-domains with nonzero  Jacobson radical. It is this class of SP-domains with which  we will be particularly concerned in the next sections.

\begin{lem} \label{jpow} Let $R$ be an SP-domain with $J(R) \ne 0$. If there is $n>0$ such that  $J(R)^n$ is principal, then $I^n$ is principal for each invertible ideal $I$ of $R$.
\end{lem}

\begin{proof} Let $J = J(R)$, and let $n >0$ be such that $J^n$ is principal, say $J^n = bR$ for  $b \in J$. Since every proper invertible ideal of $R$ is a product of invertible radical ideals, to prove that $I^n$ is principal for each invertible ideal $I$ of $R$, it suffices to show that $L^n$ is a principal ideal of $R$ for each invertible radical ideal $L$ of $R$.
Let $L$ be an invertible radical ideal of $R$. %We claim that $L = \sqrt{xR}$ for some $x \in L$.
% If $J =  L$, then $L  = \sqrt{xR}$ for any $0 \ne x \in R$, so suppose that $L \not \subseteq J$.
  Since $R/J$ is a von Neumann regular ring, the finitely generated ideal $L/J$ of $R/J$ is a principal ideal, and hence there exists $x \in L$ nonzero such that $L = xR +J$.
 Since $J$ is the Jacobson radical of $R$, a maximal ideal of $R$ contains $L$ if and only if it contains $xR$. Thus, since $R$ is one-dimensional, $L = \sqrt{xR}$.   We claim that $L^n = (x^n+b)R$, and to prove this it suffices to show that for each maximal ideal $M$ of $R$, $L^nR_M = (x^n+b)R_M$.
To this end, let $M\in\Max(R)$. Since $R$ is an almost Dedekind domain, $MR_M$ is a principal ideal of $R_M$ and hence $M^nR_M \ne M^{n+1}R_M$. Thus, since
 $M^nR_M=J^nR_M= bR_M$, it follows that
  $b\not\in M^{n+1}$ and hence $x^{n+1}+b\in M^n\setminus M^{n+1}$.
Thus, if  $L\subseteq M$,  we have $L^nR_M=M^nR_M=(x^{n+1}+b)R_M$.
%it follows that $(x^{n+1}+b)R_M=M_M^k$ for some $k\in\mathbb{N}_0$. Since $x^{n+1}+b\in M^n\setminus M^{n+1}$ we obtain that $k=n$, hence $(L^n)_M=L_M^n=M_M^n=(x^{n+1}+b)R_M$.
Otherwise, if
$L\nsubseteq M$, we have $x^{n+1}+b\not\in M$, so that $(L^n)_M=R_M=(x^{n+1}+b)R_M$.
Therefore, $L^n=(x^{n+1}+b)R$.
\end{proof}

\begin{thm} \label{BSPG3}  \label{SPG3}
Let $R$ be a domain such that  $J(R) \ne 0$. Then $R$ is an SP-domain   if and only if $R$ is a one-dimensional domain such that $J(R)$ is an invertible ideal of $R$. If also $J(R)$ is
  a principal ideal of $R$, then $R$ is a B\'ezout domain.
\end{thm}

\begin{proof} Since $J(R) \ne 0$, $R$ is not a field. If $R$ is an SP-domain, then by Lemma~\ref{SP} the radical of every nonzero principal ideal is invertible, so since $J:=J(R)$ is the radical of any nonzero ideal contained in it, $J$ is invertible.

Conversely, suppose  $R$ is one-dimensional and $J$ is invertible. For each maximal ideal $M$ of $R$, $JR_M$ is an invertible, hence principal, ideal of the local ring $R_M$. Thus, since $R_M$ is one-dimensional, $R_M$ is a DVR, which shows that $R$ is an almost Dedekind domain.
Since $J$ is contained in each maximal ideal of $R$ but not in the square of any maximal ideal,  Lemma~\ref{SP} implies
 $R$ is an SP-domain.

Finally, if $J$ is principal, then by Lemmas~\ref{SP} and~\ref{jpow}, $\sqrt{xR}$ is principal for every $x\in R$. Therefore, by Theorem~\ref{BSPG}, $R$ is a B\'ezout domain.
\end{proof}

\begin{rem} \label{pic}   {\em (1)  Let $R$ be an SP-domain with $J(R) \ne 0$. It follows from Lemma~\ref{jpow} and Theorem~\ref{BSPG3} that
 ${\rm{Pic}}(R)$ is bounded if only if
 ${\rm{Pic}}(R)$ is a torsion group, if and only if
 $J(R)^n$ is principal for some $n\in\mathbb{N}$. However, we do not know whether an SP-domain with nonzero Jacobson radical is always a B\'ezout domain; that is, we do not know whether an SP-domain with $J(R) \ne 0$ always has that  $J(R)$ is a  principal ideal.}
% It is clear that (1) implies (2), and that (2) implies (3) is also clear in light of
%Theorem~\ref{SPG}. That (3) implies (1)
% follows from Lemma~\ref{jpow}.

  {\em (2) Consider the following three properties for a one-dimensional domain $R$: (a)
$R$ is a B\'ezout SP-domain, (b)
$R$ is  radical factorial, and
(c)  every nonzero maximal ideal of $R$ contains a nonzero principal radical ideal of $R$.
It is known that (a) $\Rightarrow$ (b) $\Rightarrow$ (c); cf. \cite[Propositions 2.4(3) and 3.10(2)]{RE}. The
 converse implications do not  hold in general (even  when $R$ is a Dedekind domain and $\mbox{Pic}(R)$ is torsion-free \cite[Example 4.3]{RE}), but the counterexamples to the reverse implications have Jacobson radical $0$.
 We do
not know whether (a), (b) and (c) are equivalent when $R$ is a one-dimensional Pr\"ufer domain with nonzero Jacobson radical. }
 \end{rem}

%\begin{lem} \label{qr} If $R$ is a one-dimensional domain with nonzero Jacobson radical $J$, then for each finitely generated  ideal $I$ of $R$, there exists $x \in I$ such that  $\sqrt{I} = \sqrt{xR}$.
%\end{lem}

%\begin{proof}
% If $I \subseteq J$, then $\sqrt{I} = J = \sqrt{xR}$ for any $0 \ne x \in R$, so suppose that $I \not \subseteq J$.
 % Since $R/J$ is a von Neumann regular ring, the finitely generated ideal $(I+J)/J$ of $R/J$ is a principal ideal, and hence there exists $x \in I$ such that $I+ J  = xR +J$.  Using the fact that $R$ is a one-dimensional domain and $J$ is the Jacobson radical of $R$, we have $\sqrt{I} = I +J = xR + J = \sqrt{xR}$.
%\end{proof}

%\begin{rem} \label{spqr} {\em If $R$ is an SP-domain with nonzero Jacobson radical, then $R$ is a QR-domain.
% Since $R$ is a Pr\"ufer domain, this is immediate from Lemma~\ref{qr}.
%\end{proof}

%Observe that Corollary~\ref{SPG} is an immediate consequence of Proposition~\ref{SPG}.

\section{Topological aspects of  SP-domains}

In this section we extend ideas from \cite{WM, Olb} to
give an explicit function-theoretic  representation of $\Inv(R)$ and $\Div(R)$  for SP-domains  with nonzero Jacobson radical. We show in particular that these groups, and hence also the group $\Div(R)/\Inv(R)$, are topological invariants of the maximal spectrum $\Max(R)$ of $R$.  When $R$ is any one-dimensional domain (e.g. $R$ is an SP-domain) with nonzero Jacobson radical, then $\Max(R)$ is a {\it Boolean space} (or {\it Stone space}), that is, a compact Hausdorff space having an open basis consisting of  clopen ($=$ closed and open) subsets \cite[p.~198]{Joh}. By Stone Duality, the Boolean spaces are precisely the topological spaces that occur as the space of ultrafilters of a Boolean algebra; see for example \cite{Joh,KO}. Specifically, every Boolean space is homeomorphic to the space of ultrafilters on the Boolean algebra of its clopen sets.

We associate to each Boolean space $X$ the $\ell$-group
 $C(X,{\mathbb{Z}})$ consisting of the  continuous functions
from $X$ to $\mathbb{Z}$. (This group  is the {\it Boolean power} of ${\mathbb{Z}}$ over $X$; see~\cite{BMMO, Rib} for more on this point of view.)
 Specifically,  $(C(X,\mathbb{Z}), +)$ is an
abelian group with respect to pointwise addition; that is, for
each $f,g\in C(X,\mathbb{Z})$, $(f+g)(x)=f(x)+g(x)$ for all
$x\in X$. The group $C(X,\mathbb{Z})$ is lattice-ordered with
respect to the pointwise ordering given by $f\leq g$ if and only if $f(x)\leq
g(x)$  for all $x\in X$. Join and meet are defined, respectively, by $(f\vee
g)(x)= \max\{f(x), g(x)\}$ for all $x\in X$ and $(f\wedge g)(x)=
\min\{f(x), g(x)\}$ for all $x\in X$. Since $X$ is compact, it is straightforward to see that every function $f \in C(X,{\mathbb{Z}})$ may be represented as $f = a_1\chi_{A_i} + \cdots + a_n\chi_{A_n}$, where each $a_i \in {\mathbb{Z}}$, $\{A_1,A_2,\ldots,A_n\}$ is a partition of $X$ into clopens $A_i$, and
$\chi_{A_i}$ represents the characteristic function of $A_i$ (see for example \cite[Lemma 2.1]{BMMO} or \cite[Lemma 3.3]{Olb}).

In \cite[Section 3]{Olb}, it is shown that for each Boolean space $X$, there is a B\'ezout SP-domain with nonzero Jacobson radical whose group of divisibility is isomorphic as an $\ell$-group to $C(X,{\mathbb{Z}})$. McGovern \cite[p.~1781]{WM} has shown that this fact characterizes such rings. Namely,    a B\'ezout domain $R$ is an SP-domain with $J(R) \ne 0$  if and only if its group of divisibility ${\rm Prin}(R)$ is isomorphic as an $\ell$-group to $C(X,{\mathbb{Z}})$ for some Boolean space $X$.  In the next theorem we extend this result to the group of invertible fractional ideals of a Pr\"ufer domain.   Implicit in the proof is an argument that a Pr\"ufer domain $R$ is an SP-domain if and only if the Nagata function ring $R(T)$ of $R$ is an SP-domain.

%The next theorem shows that whether a Pr\"ufer domain
%$R$ is an SP-domain with nonzero Jacobson radical depends only the $\ell$-group ${\rm Inv}(R)$.

\begin{thm}\label{characterization} A Pr\"ufer domain $R$ is an SP-domain with $J(R) \ne 0$
if and only if there is a Boolean space $X$ (necessarily homeomorphic to $\Max(R)$) such that ${\rm Inv}(R) \cong C(X,\mathbb{Z})$ as $\ell$-groups.
\end{thm}

\begin{proof} Since the maximal ideals of the Nagata function ring $R(T)$ are of the form $MR(T)$ with $M$ a maximal ideal of $R$ and $MR(T)\cap R=M$, it follows that $J(R(T))\cap R=J(R)$. Every ideal of $R(T)$ is extended from $R$ \cite[p.~174]{FL}, and thus $J(R(T)) = 0$ if and only if $J(R) = 0$.

Suppose that $R$ is an SP-domain with $J(R) \ne 0$. Since $R$ is an SP-domain, $R$ is almost Dedekind and so is $R(T)$ \cite[Proposition 36.7]{Gil}. We show that $R(T)$ is an SP-domain via Lemma~\ref{SP}(2).
Let $M$ be a maximal ideal of $R$.  Then there exists a finitely generated ideal $I$ contained in $M$ such that $I$ is not contained in the square of any maximal ideal of $R$.
 If $IR(T)$ is contained in $N^2R(T)$ for some maximal ideal $N$ of $R$, then $I \subseteq N^2R(T) \cap R = N^2$, where the last equality holds since $R(T)$ is a faithfully flat extension of $R$. However, $I$ was chosen not in the square of any maximal ideal of $R$, so this contradiction
 implies by Lemma~\ref{SP} that $R(T)$ is an SP-domain.
 Now $R(T)$ is a B\'{e}zout SP-domain with $J(R(T)) \ne 0$, so the group of divisibility $\mbox{Prin}(R(T))$ of $R(T)$ is isomorphic as an $\ell$-group to $C(X,\mathbb{Z})$ for $X = \Max(R)$ \cite[p.~1781]{WM}. By Lemma \ref{invertible}, $\mbox{Inv}(R)$ is $\ell$-isomorphic to $C(X,\mathbb{Z})$.

 Conversely, suppose that ${\rm Inv}(R)$ is $\ell$-isomorphic to $C(X,{\mathbb{Z}})$
 for some Boolean space $X$. By Lemma \ref{invertible}, $\mbox{Prin}(R(T))$ is $\ell$-isomorphic to $C(X,\mathbb{Z})$. Hence $R(T)$ is an SP-domain with $J(R(T)) \ne 0$ \cite[p.~1781]{WM}.  Thus $J(R) \ne 0$ by the above observation that $J(R(T)) = 0$ if and only if $J(R) = 0$. Since $R(T)$ is almost Dedekind,  $R$ is almost Dedekind \cite[Proposition 36.7]{Gil}.  To see that $R$ is an SP-domain, let $M$ be a maximal ideal of $R$. By Lemma~\ref{SP} there exists a finitely generated ideal $B$ of $R(T)$ such that $B \subseteq MR(T)$ and $B$ is not contained in the square of any maximal ideal of $R(T)$.  Now, since $R$ is a Pr\"ufer domain, $B = AR(T)$ for some finitely generated ideal $A$ of $R$ \cite[p.~174]{FL}, so $A$ is contained in $M$ but not in the square of any maximal ideal of $R$. Therefore, by Lemma~\ref{SP}, $R$ is an SP-domain.

 Finally, it is routine to check that every prime $\ell$-subgroup of $C(X,{\mathbb{Z}})$ is both minimal and maximal, and that $X$ is homeomorphic to the space of prime subgroups of $C(X,{\mathbb{Z}})$; see for example
  \cite[Section 3]{Olb} (the group $G$ there is
  $C(X,{\mathbb{Z}})$).  Since the space of minimal prime $\ell$-subgroups of the group of divisibility of a B\'{e}zout domain is homeomorphic to the maximal spectrum of the domain \cite[Lemma 5.7]{Lel}, we conclude that $X$ is homeomorphic to $\Max(R(T))$, and hence by \cite[(a),  p.~174]{FL} homeomorphic to $\Max(R)$.
 \end{proof}

Thus the group of invertible ideals of an SP-domain $R$ with nonzero Jacobson radical, and hence much of the ideal theory of $R$,   depends solely on the topological space $\Max(R)$.  In \cite[Section 4]{Olb} a number of examples are given to show how the topology of $\Max(R)$ influences the SP-domain $R$.

Continuing in this line, we show next that $\Div(R)$ is also a topological invariant of the space $\Max(R)$.  For this we recall the notion of the Gleason cover of a compact Hausdorff space.
A topological  space $X$ is \textit{extremally disconnected} if the closure of any open subset of $X$ is clopen; i.e., the regular open subsets coincide with the clopen subsets of $X$.
Gleason \cite{GLE} has shown that the extremally disconnected Boolean spaces are the projective objects in the category of compact Hausdorff spaces. In particular,
 for any compact Hausdorff space $X$, there is a continuous irreducible  (i.e., sends proper closed sets to proper closed sets) surjection $j: EX \rightarrow X$, where $EX$ is an extremally disconnected Boolean space. This property characterizes $EX$ up to homeomorphism in the category of compact Hausdorff spaces $X$ \cite[Theorem 3.2]{GLE}. The space $EX$ is called the \textit{Gleason cover} of $X$.

When $X$ is a Boolean space, then $C(EX,\mathbb{Z})$ is a complete $\ell$-group \cite[Proposition 3.29]{KM}. Moreover,
 the map $\psi : C(X,\mathbb{Z})\rightarrow C(EX,\mathbb{Z})$
defined by $\psi(f)=f\circ j$ for each $f \in C(X,{\mathbb{Z}})$ is an $\ell$-group
homomorphism that is injective  since $j$ is surjective. We show  in the next lemma that under this embedding $C(EX,\mathbb{Z})$ is the completion of $C(X,{\mathbb{Z}})$.  %$\psi$ is injective.
%Note that as $X$ is a compact Hausdorff space,
%$C(X,\mathbb{Z})$ is precisely the set of bounded continuous
%function from $X$ to $\mathbb{Z}$.  It follows that $C(X,\mathbb{Z})$ is
%archimedean \cite[Example 1.6]{WM} and therefore has a completion in the sense of
%Lemma \ref{completion}.

\begin{lem}\label{completion-Boolean}
If $X$ is a Boolean space, then
the mapping $\psi:C(X,\mathbb{Z})\rightarrow C(EX,\mathbb{Z})$ is a dense embedding of $\ell$-groups and $C(EX,\mathbb{Z})$ is the completion of the image of $C(X,{\mathbb{Z}})$.

\end{lem}

\begin{proof}  As noted before the lemma, $C(EX,{\mathbb{Z}})$ is a complete $\ell$-group. By \cite[Theorem 2.4]{CA}, it suffices to show that for each $0 < f \in C(EX,{\mathbb{Z}})$, there exist $g_1,g_2 \in C(X,{\mathbb{Z}})$ such that $0 < \psi(g_1) \leq f \leq \psi(g_2)$.   Let $0 < f \in C(EX,{\mathbb{Z}})$. As discussed at the beginning of the section, there exist $n_1,\ldots,n_l\in\mathbb{N}$ and nonempty disjoint clopen sets $B_1,\ldots,B_l$ of $EX$ such that $f=\sum_{i=1}^ln_i\chi_{B_i}$.  Let $n = \max\{n_1,\ldots,n_l\}$, and let $g_2 = n \chi_X$. Then $f \leq \psi(g_2) = n\chi_{EX}$. Thus it remains to show there exists $g_1 \in C(X,{\mathbb{Z}})$ such that $0 < g_1 \leq f$.

Since $f>0$ and the $B_i$ are disjoint, there is $1\leq i\leq l$ such that $n_i\chi_{B_i}>0$.  Thus to complete the proof it suffices to show that if $B$ is a  nonempty clopen subset of $EX$, there exists a nonempty clopen subset $A$ of $X$ such that $0 < \psi(\chi_A) \leq \chi_B$. Let $j$ be the continuous irreducible surjection $EX \rightarrow X$.  Since $\psi(\chi_A) = \chi_{j^{-1}(A)}$,
%to prove that
%$\psi$ is dense,
it suffices to show there is a nonempty clopen subset $A$ of $X$ such that $j^{-1}(A) \subseteq B$. Before proving this, we show
 that $B = \iz  j^{-1}(j(B))$.

Since $B \subseteq j^{-1}(j(B))$ and $B$ is open in $EX$, we have $B \subseteq \iz j^{-1}(j(B))$. Hence $j(B) \subseteq j(\iz j^{-1}(j(B))) \subseteq j(j^{-1}(j(B))) = j(B)$.
%where the last equality follows from the fact that
%since $j$ is surjective, $C = j(j^{-1}(C))$ for all subsets $C$ of $X$.
  Thus $j(B) = j(\iz j^{-1}(j(B))).$
 %\subseteq \iz j^{-1}(j(B))$, where the last containment follows from the fact that $j$ is continuous.
%Using the fact that $j$ is continuous, we have $j^{-1}(\iz j(B)) \subseteq \iz j^{-1}(j(B))$.  Also, since $j$ is surjective, $C = j(j^{-1}(C))$ for all subsets $C$ of $X$.
%Thus
%$$\iz j(B) = j(j^{-1}(\iz j(B))) \subseteq j( \iz j^{-1}(j(B))) \subseteq j(j^{-1}(j(B)))= j(B).$$
Now, since $j$ is irreducible and $B$ is regular closed (in fact, clopen), $j(B)$ is a regular closed subset of $X$ \cite[Theorem 6.5(d), pp.~454--455]{PW}.  %Therefore, $\overline{\iz j(B)} = j(B)$.  However,
Thus $j^{-1}(j(B))$ is closed in $EX$, so that  $\iz j^{-1}(j(B))$ is a regular open, hence clopen, subset of $EX$. Therefore, $j$ maps both of the clopen sets $B$ and $\iz j^{-1}(j(B))$ onto $j(B)$.  Since $j$ is irreducible, there exists a unique clopen of $EX$ mapping onto the regular closed set $j(B)$ \cite[Theorem 6.5(d), pp.~454--455]{PW}. Consequently, $B = \iz  j^{-1}(j(B))$, which proves the claim.

Finally,
since $B$ is clopen in $EX$ and $j$ is irreducible, $\iz j(B)$ is nonempty \cite[Lemma 6.5(b), p.~452]{PW}.  Since $X$ has a basis of clopens, there exists a nonempty clopen $A$ in $\iz j(B)$.  Since  $A \subseteq j(B)$ and $A$ is open, we have   $j^{-1}(A) \subseteq \iz j^{-1}(j(B)) = B$, which completes the proof.
\end{proof}

\begin{thm} \label{EX cor}  Let $R$ be an $SP$-domain with $J(R) \ne 0$, and let $X = \Max(R)$. Then there is a commutative diagram,
\smallskip

\begin{center}
$\begin{CD}
\Inv(R)  @>{\subseteq}>> \Div(R) \\
@VV{\alpha}V        @VV{\beta}V\\
C(X,{\mathbb{Z}})    @>\psi>>  C(EX,{\mathbb{Z}}),
\end{CD}$
\end{center}
\smallskip
where the vertical maps are isomorphisms.
\end{thm}

\begin{proof}
By Theorem~\ref{characterization}, there is an isomorphism $\alpha: \Inv(R) \rightarrow C(X,{\mathbb{Z}})$, and by  Lemma \ref{completion-Boolean} the mapping $\psi:C(X,{\mathbb{Z}}) \rightarrow C(EX,{\mathbb{Z}})$ is a dense embedding, with $C(EX,{\mathbb{Z}})$ a complete $\ell$-group. By Proposition \ref{completely integrally prufer}, $\mbox{Div}(R)$ is the completion of $\mbox{Inv}(R)$, so
the mapping $\alpha$ lifts to a (unique) isomorphism $\beta:\Div(R) \rightarrow C(EX,{\mathbb{Z}})$ \cite[Theorem 1.1]{CA}.
\end{proof}

Combining the theorem with Theorem~\ref{dense subring}, we have the following ring-theoretic analogue of Lemma~\ref{completion-Boolean}.

\begin{cor}
An SP-domain $R$ with $J(R) \ne 0$ is
 a dense subring of
 a pseudo-Dedekind B\'ezout SP-domain $S$  with $J(S) \ne 0$.
\end{cor}

\begin{proof}
 By Theorem~\ref{dense subring},  $R$ is a dense subring of
 a pseudo-Dedekind B\'ezout domain $S$ with
  ${\rm Inv}(S) \cong {\rm Div}(R)$ as $\ell$-groups.  By Theorem~\ref{EX cor}, ${\rm Inv}(S)$ is isomorphic as an $\ell$-group to $C(EX,{\mathbb{Z}})$, where $X = \Max(R)$. Thus the result of McGovern \cite[p.~1781]{WM} discussed before Theorem~\ref{characterization} shows that $S$ is an SP-domain with $J(S) \ne 0$.
%(4) $\Rightarrow$ (1)  Since $S$ is   faithfully flat over $R$,  the Going Down property holds for the extension $R \subseteq S$ and every maximal ideal of $R$ has a maximal ideal of $S$ lying over it. Thus since $S$ is    a one-dimensional domain, so is $R$. Therefore, to show that $R$ is an SP-domain, it suffices to verify Lemma~\ref{SP}(2). Let $M$ be a maximal ideal. Let $N$ be a maximal ideal of $S$ lying over $M$. Then by Lemma~\ref{SP} applied to $S$ there exists a finitely generated ideal $J$ of $S$ such that $J \subseteq N$ and $J$ is not contained in the square of any maximal ideal of $S$.  Since $R$ is dense in $S$, there exists a proper finitely generated ideal $I$ of $R$ such that $J \subseteq IS$.  Let $L$ be a maximal ideal of $R$
\end{proof}

Bergman \cite[Theorem 1.1]{Ber} has proved that every group of the form $C(X,{\mathbb{Z}})$, with $X$ a Boolean space, is a free abelian group, so from
 Theorems~\ref{characterization} and~\ref{EX cor} we obtain the following corollary.

\begin{cor} \label{Bergman} If $R$ is an SP-domain with $J(R) \ne 0$, then ${\rm Inv}(R)$ and ${\rm Div}(R)$  are free abelian groups. \qed
\end{cor}

%\begin{proof} A group of the form $C(X,{\mathbb{Z}})$, with $X$ a Boolean space, is a free abelian group \cite[Theorem 1.1]{Ber}, so the corollary follows from Theorem~\ref{characterization} and Theorem~\ref{EX cor}.
%\end{proof}

%\begin{rem} \label{l remark}
%{\em The group of (necessarily invertible) nonzero ideals of a Dedekind domain is also free. It is only with the lattice ordering on the group of invertible ideals, via Theorem~\ref{characterization}, that an SP-domain  can be distinguished from a Dedekind domain. The group of (necessarily invertible) nonzero ideals of a Dedekind domain is isomorphic to a cardinal sum of copies of ${\mathbb{Z}}$. For an SP-domain this is true only when  the domain has finitely many maximal ideals and hence is a Dedekind domain. }
%\end{rem}

In light of  the fact that a completely integrally closed domain $R$ is  pseudo-Dedekind if and only if $\Inv(R)= \Div(R)$,  the group $\Div(R)/\Inv(R)$ can be viewed as a measure of how far the ring $R$ is from being pseudo-Dedekind.
We next examine this group for SP-domains. Let $n$ be an integer with $n > 1$.  A group $G$ is {\it $n$-divisible}  if for each $g \in G$, there exists $h \in G$ such that $g = nh$.
%\begin{lem} \label{weak-divisor}Let $R$ be a completely integrally closed Pr\"{u}fer domain such that its Kronecker function ring $R^b$ is also completely integrally closed. Then the weak divisor class group $Cl_w(R)$ of $R$ is isomorphic to the divisor class group $Cl(R^b)$ of $R^b$.
%\end{lem}
%\begin{proof} This is just a consequence of Lemma \ref{invertible} and Corollary \ref{completely-integrally}.
%\end{proof}

\begin{thm}  \label{torsion-free}
Let $R$ be an SP-domain with $J(R) \ne 0$. Then
 ${\rm Div}(R)/{\rm Inv}(R)$ is a torsion-free group, and the following statements are equivalent.

 \begin{itemize}
\item[(1)] ${\rm Div}(R) = {\rm Inv}(R)$ (and hence $R$ is a pseudo-Dedekind domain).

\item[(2)]  $\Max(R)$ is an extremally disconnected space.

\item[(3)] ${\rm Div}(R)/{\rm Inv}(R)$ is an $n$-divisible group for some integer $n>1$.
\end{itemize}
\end{thm}

\begin{proof}
We first prove  ${\rm Div}(R)/{\rm Inv}(R)$ is torsion-free. By  Theorem~\ref{EX cor}, to prove that
 ${\rm Div}(R)/{\rm Inv}(R)$ is torsion-free, it is enough to show that with $\psi : C(X,\mathbb{Z})\rightarrow C(EX,\mathbb{Z})$, the group $G = $ coker$(\psi)$ is torsion-free.
 Let $j$ be the continuous irreducible surjection $EX\rightarrow X$,
 %
%From the injection map $\psi : C(X,\mathbb{Z})\rightarrow C(EX,\mathbb{Z})$, we identify each $f\in C(X,\mathbb{Z})$ with $f\circ j$ in $C(EX,\mathbb{Z})$, where we recall here that $j$ is the continuous surjection $EX\rightarrow X$.
and let $f\in C(EX,\mathbb{Z})$. Suppose there is an integer $n>0$ such that $nf \in $ Im$(\psi)$.  Then for some nonempty clopen subsets $A_i$ of $X$, $i=1,\ldots,l$, that form a partition of $X$, we can write $nf=\sum_{i=1}^lm_i(\chi_{A_i}\circ j).$
%= \sum_{i=1}^lm_i \chi_{j^{-1}(A_i)}.$$
Let $k\in \{1,\ldots,l\}$, and let $x\in A_k$. Then there is  $y\in EX$ such that $x=j(y)$. Thus $$nf(y)=\sum_{i=1}^lm_i(\chi_{A_i}\circ j)(y)=\sum_{i=1}^lm_i\chi_{A_i}(x)=m_k.$$ This shows that $n$ divides $m_k$ for each $k=1,\ldots,l$. Therefore, $$nf=n\sum_{i=1}^l\frac{m_i}{n}(\chi_{A_i}\circ j) \in n \cdot {\rm Im}(\psi).$$ Since $C(EX,\mathbb{Z})$ is torsion-free, $f \in {\rm Im}(\psi)$, which proves that $G$ is torsion-free.

%It is then clear that
%a completely integrally closed domain is pseudo-Dedekind if and only if ${\rm Div}(R)/{\rm Inv}(R)$ is trivial.
Next, we show that (1), (2) and (3) are equivalent. Let $X = \Max(R)$.
Then $\psi:C(X,{\mathbb{Z}}) \rightarrow C(EX,{\mathbb{Z}})$ is surjective  if and only $j:EX\rightarrow X$ is a homeomorphism. (Recall that as discussed in the proof of Theorem~\ref{characterization}, the space of minimal prime $\ell$-subgroups of $C(Y,{\mathbb{Z}})$ is homeomorphic to $Y$ for any Boolean space $Y$.)
The equivalence of (1) and (2) now follows from
 Theorem \ref{EX cor}.
%${\rm Inv}(R) \cong C(X,\mathbb{Z})$ and ${\rm Div}(R) \cong C(EX,\mathbb{Z})$ as $\ell$-groups.
%
% is isomorphic as an $\ell$-group
%to $C(X,\mathbb{Z})$ for the Boolean space $X = \Max(R)$. By Theorem~\ref{EX cor}, ${\rm Div}(R)$ and $C(EX,\mathbb{Z})$ are isomorphic as $\ell$-groups.
%The isomorphisms involved are natural, so the equivalence of (1) and (2) follows from the fact that $\psi:C(X,{\mathbb{Z}}) \rightarrow C(EX,{\mathbb{Z}})$ is surjective  if and only $j:EX\rightarrow X$ is a homeomorphism. (Recall that as discussed in the proof of Theorem~\ref{characterization}, the space of minimal prime $\ell$-subgroups of $C(Y,{\mathbb{Z}})$ is homeomorphic to $Y$.)
Thus
  to prove the theorem,
%since if $X=EX$, i.e., $X$ is extremally disconnected if and only if $C(EX,\mathbb{Z})/C(X,\mathbb{Z})$ is trivial,
 it is enough to prove that
$G = $ coker$(\psi)$  is $n$-divisible for some integer $n>1$ if and only if $X$ is extremally disconnected.

If $X$ is extremally disconnected, then $G$ is $n$-divisible for every $n>1$ as it is
trivial. Now suppose that $G$ is $n$-divisible for some $n>1$ and $X$ is not extremally disconnected. Then there exists a regular closed subset $V$ of $X$ that is
not clopen.  Let $C = \overline{j^{-1}(\iz V)} $.
%\overline{j^{-1}(\iz V)}$.
Since $j$ is continuous, $C$ is a regular closed set. Since $EX$ is extremally disconnected,  $C$ is clopen and
$\chi_C
\in C(EX,\mathbb{Z})$.  Now, since $G$ is $n$-divisible,
there is $f\in C(EX,\mathbb{Z})$
such that $nf-\chi_{C} = h \circ j$ for some
 $h\in C(X,\Z)$.

 Let $V' = \overline{X\setminus V}$. Since
 $V$ is not clopen, there exists $x \in V \cap V'$.  The open set $h^{-1}(h(x))  \subseteq X$ is a neighborhood of $x$, so since $\iz V$ is dense in $V$, there exists $x_1 \in  \iz V$ such that $h(x_1) = h(x)$. Let $y_1 \in EX$ such that $j(y_1) = x_1$.  Then $y_1 \in j^{-1}(x_1) \subseteq j^{-1}(\iz V) \subseteq C$, so that
\begin{flushleft}
$(*) \ \  \ \ \  \ \ \ \ \ \ \ h(x)  = h(x_1)= (h\circ j)(y_1) = (nf-\chi_{C})(y_1)=nf(y_1) -1.$
\end{flushleft}
Similarly, since $X \setminus V$ is dense in $V'$, there exists $x_2 \in X \setminus V$ such that $h(x_2) = h(x)$.
 Let $y_2 \in EX$ such that $j(y_2)= x_2$.
 Then $y_2 \in j^{-1}(x_2) \subseteq j^{-1}(X \setminus V)$.
 Since $j$ is continuous and $V$ is regular closed, we have
 %\begin{flushleft}
%(*) \ \ \ \ \ \ \ \ \ \ \ \ \ \  \  \ \ \ \ \ \ \ \  \
$$C = \overline{j^{-1}(\iz V)} \subseteq j^{-1}(\overline{\iz V}) = j^{-1}(V).$$
%end{flushleft}
If also $y_2 \in C$,
 then $y_2 \in j^{-1}(V)$, so that $j(y_2) \in (X \setminus V) \cap V$, a contradiction. We conclude   $y_2 \not \in C$ and
$$h(x) = h(x_2)= (h\circ j)(y_2) = (nf-\chi_{C})(y_2)=nf(y_2).$$ Combining this observation with $(*)$, we have $nf(y_1) - 1 = h(x) =  nf(y_2)$, which implies $n(f(y_1) - f(y_2)) = 1$, a contradiction to the fact that  $n>1$ and $f(y_1)-f(y_2) \in {\mathbb{Z}}$.
Therefore,
 $G$ is not $n$-divisible.
\end{proof}

\begin{rem} {\em
Let $X$ be a Boolean space. The proof of
 Theorem~\ref{torsion-free} shows that the cokernel
 of the embedding $C(X,{\mathbb{Z}}) \rightarrow C(EX,{\mathbb{Z}})$ is torsion-free, and it is $n$-divisible for some $n>1$ if and only if $X$ is extremally disconnected.   We do not know (a)~whether  the cokernel $G(X)$ of the embedding $C(X,{\mathbb{Z}}) \rightarrow C(EX,{\mathbb{Z}})$ is free, or (b) whether there exists a Boolean space $X$ such that $G(X)$ is finitely generated and nontrivial. With regards to (a),  the groups $C(X,{\mathbb{Z}})$ and $C(EX,{\mathbb{Z}})$ are free by the theorem of Bergman \cite[Theorem 1.1]{Ber} cited above.
 }
\end{rem}

%By Corollary~\ref{Bergman}, if $R$ is an SP-domain with $J(R) \ne 0$, then the groups $\Inv(R)$ and $\Div(R)$ are free. We do not know however whether (a) the group $\Div(R)/\Inv(R)$ is free, or (b) whether $\Div(R)/\Inv(R)$ can be finitely generated but nonzero. By the previous remark, these questions can be viewed as asking (a) whether for a Boolean space $X$, the cokernel $G(X)$ of the embedding $C(X,{\mathbb{Z}}) \rightarrow C(EX,{\mathbb{Z}})$ is free, and (b) whether there exists a Boolean space $X$ such that $G(X)$ is finitely generated and nontrivial.

\section{Cantor-Bendixson theory for one-dimensional Pr\"ufer domains}

A one-dimensional Pr\"ufer domain $R$  is a Dedekind domain if and only if every maximal ideal of $R$ is finitely generated. Thus if $R$ is not Dedekind, then there exist maximal ideals that are not finitely generated. More generally, a maximal ideal of a one-dimensional Pr\"ufer domain need not even be the radical of a finitely generated ideal. For example, if $R$ is an almost Dedekind domain, then it is routine to see that a maximal ideal is finitely generated if and only if it contains a finitely generated ideal that is not contained in any other maximal ideal of $R$.  Following \cite{LL}, we say a maximal $M$ of a Pr\"ufer domain $R$ is {\it sharp} if it contains a finitely generated ideal that is not contained in any other maximal ideal; otherwise, $M$ is {\it dull}. Thus a maximal ideal of a one-dimensional Pr\"ufer domain is sharp if and only if it is the radical of a finitely generated ideal, and a maximal ideal of an almost Dedekind domain is sharp if and only if it is finitely generated.

Recently, Loper and Lucas \cite{LL} have introduced the notions of sharp and dull degrees of maximal ideals of one-dimensional Pr\"ufer domains as a measure of how far a maximal ideal is from being sharp.
 In this section we recast their theory in topological terms and use Cantor-Bendixson theory to shed additional light on the sharp and dull degrees of maximal ideals.
 We  show  in particular that the sharp and dull degrees of one-dimensional Pr\"ufer domains are topological invariants of their maximal spectra when viewed with the inverse topology. (When also the Pr\"ufer domain has nonzero Jacobson radical, this topology coincides with the Zariski topology on $\Max(R)$.)  This allows us to prove existence results for almost Dedekind domains of various sharp and dull degrees. These results give a different approach, as well as  extend, similar existence theorems due to  Loper and Lucas.  In our context these existence results  reduce to quick consequences of well-known facts about Boolean spaces.

Loper and Lucas' definition of sharp and dull degree is motivated by a characterization of sharp maximal ideals:
A maximal ideal $M$ of a Pr\"ufer domain $R$ is sharp if and only if $\bigcap_{N \ne M}R_N \not \subseteq R_M$, where $N$ ranges over the maximal ideals of $R$ distinct from $M$ \cite[Corollary 2.]{GH}; i.e., $M$ is sharp if and only if $R_M$ is irredundant in the representation of $R$ as the intersection of its localizations at maximal ideals. While $R_M$ may not be irredundant in this representation, it could be irredundant in representations of overrings of $R$, and it is this observation that provides the motivation for Loper and Lucas' definition of sharp and dull degrees of a one-dimensional Pr\"ufer domain. (Alternatively, while $M$ may not be sharp in $R$, it can extend to a sharp maximal ideal of an overring of $R$.)
 We extend their definition using transfinite induction to permit infinite degree also.

\begin{defi}
Let $R$ be a Pr\"ufer domain.
%A prime ideal $P$ of  $R$ is {\it sharp} if contains a finitely generated ideal $I$ such that that the only maximal ideals containing $I$ are those that contain $P$. If $P$ is not  sharp, then $P$ is a {\it dull} prime ideal.
Denote by $\Ms(R)$ the set of sharp maximal ideals of $R$  and by $\Md(R)$ the set of dull maximal ideals of $R$.
Define $R_0=R$. For each ordinal number $\alpha$, define $R_{\alpha+1} = \bigcap_{M \in \Md(R_\alpha)} R_M$, and for each limit ordinal number $\lambda$, define $R_{\lambda} = \bigcup_{\alpha < \lambda}R_\alpha$.
 \begin{itemize}
 \item[(1)] The ring $R$ has  {\it sharp degree} $\alpha$ if $R_{\alpha+1}$ is the quotient field  $F$ of $R$ and $R_\alpha \ne R_{\alpha+1}$; equivalently, $R$ has sharp degree $\alpha$ if $\Md(R_\alpha) = \emptyset$ and $\Md(R_\beta) \ne \emptyset$ for all $\beta < \alpha$. We say $R$ {\it has a sharp degree} if $R$ has sharp degree $\alpha$ for some $\alpha$.

 \item[(2)]
 The ring $R$ has {\it dull degree} $\alpha$ if $R_{\alpha} = R_{\alpha+1} \ne F$ and $R_{\beta} \ne R_\alpha$ for all $\beta < \alpha$. In this case, we say $R$ {\it has a dull degree}.
 \end{itemize}
\end{defi}

% Define $R_1 = \bigcap_{M \in \Md(R)}R_M$, and for each $i\geq 1$, define $R_{i+1} = \bigcap_{M \in \Md(R)}(R_i)_M$.  Following \cite{LL}, the {\it sharp degree} of $R$ is $n$ if $R_{n+1}$ is the quotient field  $F$ of $R$ and $R_n \ne R_{n+1}$, while $R$ has {\it dull degree} $n$ if $R_{n} = R_{n+1} \ne F$ and $R_{n-1} \ne R_n$.

Our main result of this section, Theorem~\ref{connect}, connects sharp and dull degrees with the Cantor-Bendixson rank of the topological space $\Max(R)$ endowed with the inverse topology. We recall this topology.

\begin{defi} Let $R$ be a ring. We denote by $\Maxi(R)$ the topological space whose underlying set is $\Max(R)$ and which has a basis of open sets of the form $\{M \in \Max(R):r_1,\ldots,r_n \in M\}$, where $r_1,\ldots,r_n \in R$; equivalently, the quasicompact open subsets of the space $\Max(R)$ (with respect to the Zariski topology) form a basis of closed of sets for $\Maxi(R)$.  This topology on $\Maxi(R)$ is  the {\it inverse topology}.
\end{defi}

\begin{lem} \label{connection} Let $R$ be a one-dimensional Pr\"ufer domain.
\begin{itemize}
\item[(1)] Let $X$ be a subspace of $\Maxi(R)$. Then $M$ is a limit point for $X$ if and only if every finitely generated ideal of $R$ contained in $M$ is contained in another maximal ideal in $X$.

\item[(2)] A maximal ideal $M$ of $R$ is sharp if and only if $M$ is an isolated point in $\Maxi(R)$; $M$ is dull if and only if $M$ is a limit point in $\Maxi(R)$.

\item[(3)] If $I$ is a nonzero proper ideal of $R$, then  the Zariski and inverse topologies on $\Spec(R/I)$ coincide.

\item[(4)] The Zariski and inverse topologies on $\Max(R)$ coincide if and only if $J(R)\ne 0$.
\end{itemize}
\end{lem}

\begin{proof}
(1)  Suppose there is a  finitely generated ideal $I$ of $R$ contained in $M$ but in no other maximal ideal in $X$. Then  $\{M\} = \{N \in X \cup \{M\}:I \subseteq N\}$ is open in $X \cup \{M\}$, so $M$ is an isolated point in $X \cup \{M\}$. Conversely, if
$M$ is an isolated point in $X \cup \{M\}$, then there is a finitely generated ideal $J$ of $R$ such that $\{M\} = \{N \in X \cup \{M\}:J \subseteq N\}$.

(2) This follows from (1).

(3) Without loss of generality we assume that $I$ is a radical ideal of $R$.  For each ideal $K$ of $R$, let $V_I(K) = \{P \in \Spec(R):I+K  \subseteq P\}$ and $U_I(K) = \{P \in \Spec(R):I \subseteq P$ and $K \not \subseteq P\}$.
To prove (3), it suffices to show that for each finitely generated ideal $J$ of $R$, there is a finitely generated ideal $K$ of $R$ such that  $V_I(J)  = U_I(K)$.
 Let $J$ be a finitely generated ideal of $R$.
Since $I$ is a nonzero radical ideal and $R/I$ has Krull dimension $0$, $R/I$ is a von Neumann regular ring. Thus $(J+{I})/{I}$ is a summand of $R/I$, so that  there exists an ideal $L$ of $R$ containing $I$ such that $(J+I) \cap L = I$ and $J+  L = R$.  This implies there is a finitely generated ideal $K \subseteq L$ such that $J \cap K \subseteq I$ and $J + K = R$. Therefore, $V_I(J)= U_I(K)$, which proves the claim.

(4) Suppose that the Zariski and inverse topologies on $\Max(R)$ coincide. Note that $\Max(R)$ is compact with respect to the Zariski topology. Therefore, $\Maxi(R)$ is compact. Since $\Max(R)=\bigcup_{a\in R\setminus\{0\}}\{M\in\Max(R):a\in M\}$, we infer that $\Max(R)=\bigcup_{a\in E}\{M\in\Max(R):a\in M\}$ for some finite nonempty $E\subseteq R\setminus\{0\}$. Observe that $0\ne\prod_{a\in E} a\in J(R)$.

To prove the converse, suppose that $J(R)\ne 0$. By (3) it follows that the Zariski and inverse topologies on $\Spec(R/J(R))$ coincide, and hence the Zariski and inverse topologies on $\Max(R)$ coincide.
\end{proof}

%We recall in the next definition the Cantor-Bendixson rank of a topological space.

\begin{defi}
Let $X$ be a topological space. Define $X^0 = X$, and for each ordinal number $\alpha$ let $X^{\alpha+1}$ denote the set of limit points of $X^\alpha$.   For each limit ordinal $\lambda$, let $X^{\lambda} = \bigcap_{\alpha< \lambda}X^\alpha$. The closed set $X^\alpha$ of $X$ is the $\alpha$-th {\it Cantor-Bendixson derivative} of $X$. The {\it Cantor-Bendixson rank} of $X$, $\rk(X)$, is the smallest ordinal number $\alpha$ for which $X^{\alpha} = X^{\alpha + 1}$. % For any topological space $X$, $\rk(X)$ is a countable ordinal \cite[p.~262]{Kur}.
The topological space $X$ is {\it scattered} if $X^{\rk(X)} = \emptyset$ (equivalently, every nonempty subspace of $X$ contains an isolated point).
\end{defi}

By Lemma~\ref{connection}(2), if $X = \Max^{-1}(R)$,  the set $X^1$ of limit points of $X$ is $\Md(R)$.
Consequently, $R_1 = \bigcap_{M \in X^1}R_M$.  We extend this to all ordinals in the next lemma.
To do so, we recall that if $R$ is a one-dimensional Pr\"ufer domain and $S$ is a ring between $R$ and its quotient field such that $S$ is not a field, then $S$ is also a one-dimensional Pr\"ufer domain with $\Max(S) = \{MS:M \in \Max(R), MS \ne S\}$ \cite[Theorem 1]{GO}.

\begin{lem} \label{derivative} \label{easy lemma} Let $R$ be a one-dimensional Pr\"ufer domain with quotient field $F$, let $X= \Maxi(R)$  and let $\alpha$ be an ordinal number. Then
\begin{itemize}
\item[{(1)}]
 $\Max(R_\alpha)\setminus\{0\} = \{MR_\alpha:M \in X^\alpha\}$,

 \item[{(2)}] $X^\alpha =  \{M \in X:MR_\alpha \ne R_\alpha\}$, and

 \item[{(3)}]
 $R_\alpha = \bigcap_{M \in X^\alpha}R_M$.
\end{itemize}
 Thus if $\beta$ is an ordinal, $X^\alpha = X^\beta$ if and only if $R_\alpha = R_\beta$.

\end{lem}

\begin{proof}
(1) The proof is by transfinite induction. Suppose (1) holds for an ordinal $\alpha$.  Let $N \in \Max(R_{\alpha+1})\setminus\{0\}$.  Then, as discussed before the lemma, $N = MR_{\alpha+1}$ for some maximal ideal $M$ of $R$.
 Thus $MR_\alpha \ne R_\alpha$, and so $MR_\alpha \in \Max(R_\alpha)\setminus\{0\}$, which
 by the induction hypothesis implies $M \in X^\alpha$.
Moreover, since the maximal ideal $MR_{\alpha+1}$ of $R_{\alpha+1}$ is extended from a maximal ideal of $R_\alpha$, $MR_{\alpha}$   is necessarily dull in $R_\alpha$ \cite[Lemma 2.1]{LL}.  Therefore, every finitely generated ideal of $R$ that is contained in $M$ is contained in some other maximal ideal of $R_\alpha$. Since   $\Max(R_\alpha)\setminus\{0\} = \{LR_\alpha:L \in X^\alpha\}$,  every finitely generated ideal of $R$ that is contained in $M$ is contained in some other maximal ideal in $X^\alpha$. By Lemma~\ref{connection}(1), $M$ is  a limit point in $X^\alpha$, so that $M \in X^{\alpha+1}$.
%Therefore, $X^{\alpha+1} = \{M \in X:MR_{\alpha+1} \ne R_{\alpha+1}\}$.
This shows that $\Max(R_{\alpha+1})\setminus\{0\} \subseteq \{MR_\alpha:M \in X^{\alpha+1}\}$.

Conversely, suppose
 $M \in X^{\alpha +1}$.  Then $M \in X^\alpha$, so  since (1) holds for $\alpha$, $MR_\alpha$ is a maximal ideal of $R_\alpha$.  Since $M$ is a limit point in $X^\alpha$,  Lemma~\ref{connection}(1) implies that every  finitely generated ideal of $R$ contained in $M$ is contained in some other maximal ideal in $X^\alpha$.  Let $I = (x_1,\ldots,x_n)R_{\alpha}$ be a finitely generated ideal of $R_\alpha$ contained in $MR_{\alpha}$, and let $J = (x_1,\ldots,x_n)R \cap R$. Then $J$ is a finitely generated ideal of $R$ \cite[Exercise 1.1, p.~95]{FS} and
 $I = JR_\alpha$ \cite[Claim (A), p.~95]{FS}.  Since $J$ is contained in $M$, $J$ is contained in some other maximal ideal in $X^\alpha$. Since (1) holds for $\alpha$, it follows that
  $I = JR_\alpha$ is contained in some other maximal ideal in $\Max(R_\alpha)$.
Therefore, $MR_\alpha$ is a dull maximal ideal of $R_\alpha$, so that $MR_{\alpha +1} \ne R_{\alpha +1}$ and hence $MR_{\alpha+1} \in \Max(R_{\alpha+1})\setminus\{0\}$.  This proves $\Max(R_{\alpha+1})\setminus\{0\} = \{MR_{\alpha+1}:M \in X^{\alpha+1}\}$, which verifies (1) for $\alpha+1$.

%To prove the reverse inclusion, let $M \in X$ such that $MR_{\alpha+1} \ne R_{\alpha +1}$.  Then $MR_\alpha \ne R_\alpha$, so by the induction hypothesis, $M \in X^\alpha$.
%Moreover,  $MR_{\alpha+1}$ is extended from a maximal ideal of $R_\alpha$, which by  \cite[Lemma 2.1]{LL} is necessarily dull in $R_\alpha$.  Therefore, every finitely generated ideal of $R$ that is contained in $M$ is contained in some other maximal ideal of $R_\alpha$. Since $\Max(R_\alpha) = \{MR_\alpha:M \in X^\alpha\}$, we have by Lemma~\ref{connection} that $M$ is a limit point in $X^\alpha$, and hence $M \in X^{\alpha+1}$.  Therefore, $X^{\alpha+1} = \{M \in X:MR_{\alpha+1} \ne R_{\alpha+1}\}$.

Next, suppose that $\lambda$ is a limit ordinal and  (1) holds for every $\alpha < \lambda$.  Let $M$ be a maximal ideal of $R$. Then $MR_{\lambda} \ne R_{\lambda}$ if and only if $MR_{\alpha} \ne R_{\alpha}$ for each $\alpha < \lambda$, if and only if (by (1)) $M \in \bigcap_{\alpha < \lambda}X^{\alpha}$, if and only if $M \in  X^{\lambda}$.  Therefore, $\Max(R_\lambda)\setminus\{0\} = \{MR_\lambda:M\in\Max(R),MR_\lambda\not=R_\lambda\} = \{MR_\lambda:M \in X^{\lambda}\}$.
This completes the induction and the proof of (1).

(2) If $M \in X^\alpha$, then, by (1), $MR_\alpha \ne R_\alpha$.  Conversely, if $M \in \Max(R)$ and $MR_\alpha \ne R_\alpha$, then, as discussed before the lemma, $MR_\alpha$ is a maximal ideal of $R_\alpha$ and hence by (1), $M \in X^\alpha$.

(3)
This follows from (1).

The final assertion follows from (2) and (3).
\end{proof}

%\begin{lem} An ideal $M \in \Max(R)$ is in $X^{\alpha+1}$ but not $X^\alpha$ if and only if $M$ extends to a dull maximal ideal in $R_\alpha$ and  a sharp maximal ideal in $R_{\alpha+1}$.
%\end{lem}

%\begin{lem} Let $R$ be a one-dimensional Pr\"ufer domain, let $X = \Maxi(R)$ and let $\alpha$ be an ordinal number. Then \begin{center} $X^{\alpha+1} = \{M \in X: MR_\alpha$ is a dull prime ideal of $R_\alpha\}.$\end{center}
%\end{lem}

%\begin{proof} Let $N \in X^{\alpha+1}$. Then $N$ is a limit point in $X^\alpha$, so that by Lemma~\ref{derivative}, $$\bigcap_{P \in {\mbox{Max}}(R_\alpha),P \ne NR_\alpha} (R_\alpha)_P = \bigcap_{M \in X^\alpha, M \ne N}R_M \subseteq R_N,$$ which implies that $NR_\alpha$ is a dull prime ideal of $R_\alpha$.

%\end{proof}

%A topological space  $X$ is {\it scattered} if every nonempty (closed) subset contains an isolated point in the subspace topology;  equivalently, $X^\alpha = \emptyset$ for some ordinal $\alpha$.

\begin{thm} \label{connect} Let $R$ be a one-dimensional Pr\"ufer domain, let $X = \Maxi(R)$ and let $\alpha$ be an ordinal number.
\begin{itemize}
\item[(1)] The ring $R$ has sharp degree $\alpha$  if and only if $ \rk(X)=\alpha+1$ and $X^{\alpha+1} = \emptyset$. Thus if $R$ has a sharp degree, then $X$ is a scattered space.
%
%is a scattered space. In this case,
% $ \rk(X)=\alpha+1$ and $X^{\alpha+1} = \emptyset$.
%In this case, $\alpha$ is a countable ordinal.

\item[(2)] If $X$ is a scattered space and $J(R)\ne 0$, then $R$ has a sharp degree.

\item[(3)] The ring $R$ has dull degree $\alpha$ if and only if  $ \rk(X) = \alpha $ and  $X^{\alpha} \ne \emptyset$. Thus   $R$ has a dull degree  if and only if $X$ is not a scattered space.

 \end{itemize}
\end{thm}

\begin{proof}
(1)  Suppose that $R$ has sharp degree $\alpha$. Then $R_\alpha \subsetneq R_{\alpha+1} = F$, so that by Lemma~\ref{easy lemma}, $X^{\alpha+1} \subsetneq X^\alpha$ and  $X^{\alpha+1} = \emptyset$. Hence $X$ is scattered and
 $\rk(X) = \alpha+1$.   %Moreover, as noted at the beginning of the section, $\rk(X)$ is a countable ordinal.
Conversely, if
$X^{\alpha+1} = \emptyset$ and $\rk(X)=\alpha+1$, then Lemma~\ref{derivative} implies that $R_\alpha  \subsetneq R_{\alpha+1} = F$, and hence $R$ has sharp degree $\alpha$.

(2) Suppose $X$ is scattered and $J(R)\ne 0$, and let $\alpha$ be the smallest ordinal such that $X^{\alpha} = \emptyset$. It follows from Lemma~\ref{connection}(4) that $X$ is compact. If $\alpha$ is a limit ordinal, then $\bigcap_{\beta  < \alpha}X^\beta = \emptyset$, contrary to the fact that $X$ is compact and the $X_\beta$ are nonempty. Therefore, $\alpha$ is a successor ordinal, say $\alpha = \beta^+$, and $\rk(X) = \alpha = \beta +1$ with $X^{{\beta}+1} = \emptyset$. Thus $R$ has sharp degree $\beta$.

(3) Suppose that  $R$ has dull degree $\alpha$. Then $R_\alpha = R_{\alpha+1} \ne F$ and $R_\beta \subsetneq R_\alpha$ for all $\beta < \alpha$.  Thus, by Lemma~\ref{derivative}, $\emptyset \ne X^{\alpha+1} = X^\alpha \subsetneq  X_\beta$ for all $\beta < \alpha$. Therefore, $\rk(X) = \alpha$. % and $X^{\alpha}$ is a perfect nonempty closed, hence compact and Hausdorff, subspace of $X$. A perfect nonempty compact Hausdorff space is uncountable,
Conversely, suppose that $\rk(X) = \alpha$ and $X^\alpha \ne \emptyset$.  By Lemma~\ref{easy lemma}, for each $\beta < \alpha$,  $R_\beta \subsetneq R_\alpha = R_{\alpha+1} \ne F$. Thus $R$ has dull degree $\alpha$. The last statement follows from the first.
%
% Now suppose that $X^{\alpha+1} \ne \emptyset$. Since $X^{\alpha} = X^{\alpha+1}$, we have by Lemma~\ref{derivative} that   $R_{\alpha} = R_{\alpha+1}$.  Moreover, since $X^{\alpha+1} \ne \emptyset$, the lemma also implies that $R_{\alpha+1} \ne F$.  Finally, if $\beta < \alpha$, then $X^\beta \supsetneq X^\alpha$. Let $M \in X^\beta \setminus X^\alpha$. Then $M$ is an isolated point in $X^\beta$, which implies that $\bigcap_{N \ne M,N \in X^\alpha}R_N \not \subseteq R_M$, and hence $R_{\beta} \not \subseteq R_M$. Thus $R_\beta \ne R_\alpha$. This shows that $R$ has dull degree $\alpha$.
\end{proof}

%\begin{cor} Let $R$ be a one-dimensional Pr\"ufer domain.
%%\item[(1)] If $R$ has sharp degree $\alpha$ for some $\alpha$, then $\alpha$ is a countable ordinal.
%\item[(2)] If $R$ has dull degree $\alpha$ for some $\alpha$, then $\Max(R)$ is an uncountable set.
%\end{itemize}
%\end{cor}

%\begin{proof} (1) Let $X = \Maxi(R)$.  By Theorem~\ref{connect}, $\rk(X) = \alpha+1$. As noted at the beginning of the section, $\rk(X)$ is a countable ordinal, so (1) is clear.

%(2) By Theorem~\ref{connect}, $X^{\alpha}$ is nonempty and $\rk(X) = \alpha$. Thus $X^{\alpha}$ is a perfect nonempty closed, hence compact and Hausdorff, subspace of $X$. A perfect nonempty compact Hausdorff space is uncountable,

%\end{proof}

Combining Theorem~\ref{connect} with  results from \cite{HO}, we see in the next corollary that the ideals of almost Dedekind domains possessing a sharp degree can be decomposed as a irredundant intersections of powers of maximal ideals.

\begin{cor}  Let $R$ be an almost Dedekind domain that is not a field. If $R$ has a sharp degree, then
every nonzero proper ideal of $R$  has a unique representation as an irredundant intersection of powers of distinct maximal ideals.

\end{cor}

\begin{proof}  Suppose that $R$ is an almost Dedekind domain with  sharp degree $\alpha$ for some ordinal $\alpha$. By Theorem~\ref{connect}, $\Maxi(R)$ is scattered, and hence so is $\Maxi(R/I)$ for every nonzero ideal $I$ of $R$.  Since Lemma~\ref{connection}(3) implies that the Zariski and inverse topologies agree on $\Max(R/I)$, we have that  $\Max(R/I)$ is scattered. It follows now from \cite[Corollary 3.9]{HO} that every nonzero proper ideal of $R$ can be represented uniquely as an irredundant intersection of completely irreducible ideals.
%
%Conversely, suppose every ideal has a unique representation as an irredundant intersection of powers of distinct maximal ideals. Then by \cite[Corollary 3.9]{HO}, $R$ is an almost Dedekind domain such that $\Max(R/I) = \Maxi(R/I)$ is scattered for each nonzero proper ideal $I$ of $R$.  Thus every proper closed subset of $\Maxi(R/I)$ is scattered, and it follows that $\Maxi(R)$ is scattered. By Theorem~\ref{connect}, $R$ has sharp degree $\alpha$.
\end{proof}

\begin{cor} \label{sharp cor} If $R$ is a one-dimensional Pr\"ufer domain such that $J(R)\ne 0$ and $X=\Maxi(R)$ is countable, then there is a countable ordinal $\alpha$ such that $R$ has sharp degree $\alpha$ and $X^\alpha$ is finite.
\end{cor}

\begin{proof} Let $X = \Maxi(R)$. Since the space $X$ is countable, compact and Hausdorff, $\rk(X)$ is
 a countable successor ordinal
  $\alpha + 1$ such that $X^{\alpha}$ is finite and $X^{\alpha+1}=\emptyset$ \cite[p.~18]{MS}.
By Theorem~\ref{connect}, $R$ has sharp degree $\alpha$. %and by Lemma~\ref{easy lemma}, $|\Max(R_\alpha)|= |X^\alpha|$, which proves the  corollary.
\end{proof}

\begin{cor} If a one-dimensional Pr\"ufer domain $R$ with $J(R)\ne 0$ has a dull degree, then $\Max(R)$ is an uncountable set.
\end{cor}

\begin{proof} If $\Max(R)$ is countable, then, by Corollary~\ref{sharp cor}, $R$ has a sharp degree, in which case $R$ does not have dull degree.
\end{proof}

In the next corollary we recover  Loper and Lucas' existence results for almost Dedekind domains with specified sharp degrees (see \cite[Section 3]{LL}). In their case they restrict to finite sharp degrees; in ours, we allow sharp degrees of any ordinal.

%In the next corollary we recover  Loper and Lucas' existence results for almost Dedekind domains with specified sharp degrees (see \cite[Section 3]{LL}). In their case they restrict to finite sharp degrees; in ours, we allow  sharp degrees of any successor ordinal.

\begin{cor} \label{ordinal exist}
Every ordinal number occurs as the sharp degree of an SP-domain.
\end{cor}

\begin{proof}
Note that every Dedekind domain that is not a field  is an SP-domain that has sharp degree $0$. Let $\alpha$ be a nonzero ordinal number, and let $\omega$ denote the first infinite ordinal. The ordinal space (hence scattered space) $X=\omega^{\alpha}+1$ satisfies $\rk(X) = \alpha+1$ by \cite[Theorem 2.6(1)]{BM}. Since $\omega^{\alpha}+1$ is a successor ordinal, the ordinal space $X$ is compact. Thus $X$ is a Boolean space, and by \cite[Section 3]{Olb} there exists an SP-domain $R$ with $J(R)\not=0$ such that $\Max(R)$ is homeomorphic to $X$. Therefore, the corollary follows from Theorem~\ref{connect}(1).
\end{proof}

%\begin{cor} \label{ordinal exist}
%An ordinal $\alpha$ occurs as the sharp degree of a one-dimensional Pr\"ufer domain (even an SP-domain) if and only if $\alpha$ is a successor ordinal.
%\end{cor}

%\begin{proof}
%By Theorem~\ref{connect}(2),  the sharp degree of a one-dimensional Pr\"ufer domain is a successor ordinal. Conversely, if $\alpha$ is a successor ordinal, then there exists an ordinal  space (hence scattered space) $X$ such that $\rk(X) = \alpha$ \cite[Theorem 2.6(2)]{BM}. Since $\alpha$ is a successor ordinal, the ordinal space $X$ is compact.  Thus $X$ is a Boolean space, and by \cite[Section 3]{Olb}, there exists an SP-domain $R$ such that $\Max(R)$ is homeomorphic to $X$. Therefore, the corollary follows from Lemma~\ref{derivative} and Theorem~\ref{connect}.
%\end{proof}

By restricting to countable ordinals, we are also able to prescribe  the size of the maximal spectrum of the penultimate ring in the sequence of the $R_\alpha$'s.

\begin{cor} \label{countable exist}
For  every countable ordinal $\alpha$ and finite positive integer $n$, there exists an SP-domain $R$ of sharp degree $\alpha$ such that $J(R) \ne 0$,  $\Max(R)$ is countable and  $|\Max(R_\alpha)| = n$.
\end{cor}

\begin{proof}  Let $\omega$ denote the first infinite ordinal.  The ordinal space $X=\omega^\alpha \cdot n + 1$ is then a countable compact Hausdorff space for which $X^{\alpha}$ has $n$ elements and $X^{\alpha+1} = \emptyset$ \cite{MS}.  Every ordinal space is totally disconnected and Hausdorff, and since the
 ordinal $\omega^\alpha \cdot n + 1$ is a successor ordinal, $X$ is compact. Thus $X$ is a Boolean space and  the proof now finishes just as the proof of Corollary~\ref{ordinal exist}.
%
% hus $X$ is a Boolean space. By \cite[Section 3]{Olb}, there exists an SP-domain $R$ such that $\Max(R)$ is homeomorphic to $X$. Therefore, the corollary follows from Lemma~\ref{derivative} and Theorem~\ref{connect}.
\end{proof}

%\begin{rem} {\em It is possible for an SP-domain to have uncountable sharp degree. Let $\omega_1$ be the first uncountable ordinal. As in the proof of Corollary~\ref{countable exist}, $X = \omega_1 + 1$ is a Boolean space. Moreover, $\rk(X) = \omega_1+1$. As in the corollary, there then exists an SP-domain $R$ having sharp degree $\omega_1$.}
%\end{rem}

\bigskip
\textbf{Acknowledgement.} The third author of this work was supported by the Austrian Science Fund FWF, Project number P26036-N26.

%\section{SP-domains arising as holomorphy rings}

%{\bf I have some results on this that could be included if we decide to go in this direction. The idea is to show how SP-domains arise when intersecting valuation rings arising from ``curves'' in Spec($D$), where $D$ is a two-dimensional Noetherian domain. Unfortunately, the ideas need a good bit of technical preliminaries and the techniques don't fit well with the rest of the paper, so we may not want to go in this direction.}


\begin{thebibliography}{aaaaa}


%\bibitem{AB} S. Abhyankar, {\it Resolution of Singularities of Embedded Algebraic Surfaces},
%Second Edition. Springer Monographs in Mathematics.
%Springer-Verlag, Berlin, 1998.
%\bibitem{DDA} D. D. Anderson and D. F. Anderson, Generalized GCD domains, Comment. Math. Univ. St. Pauli {\bf 28} (1979), 215--221.

\bibitem{AK} D. D. Anderson and B. G. Kang, Pseudo-Dedekind domains and divisorial ideals in $R[X]_{T}$, {J. Algebra} {\bf 122} (1989), no. 2, 323--336.

\bibitem{Ber} G. M. Bergman, Boolean rings of projection maps, J. London Math. Soc. (2) {\bf 4} (1972), 593--598.

\bibitem{BM} G. Bezhanishvili and P. J. Morandi, Scattered and hereditarily irresolvable spaces in modal logic, Arch. Math. Logic {\bf 49} (2010), no. 3, 343--365.

\bibitem{BMMO} G. Bezhanishvili, V. Marra, P. J. Morandi and B. Olberding, Idempotent generated algebras and Boolean powers of commutative rings, Algebra Universalis {\bf 73} (2015), no. 2, 183--204.

\bibitem{Bir} G. Birkhoff, {\it Lattice theory.} Corrected reprint of the 1967 third edition. American Mathematical Society Colloquium Publications, 25. American Mathematical Society, Providence, R.I., 1979.

\bibitem{Bou} N. Bourbaki, {\it Commutative Algebra, Chapters 1--7}. Translated from the French. Reprint of the 1989 English translation. Elements of Mathematics (Berlin). Springer-Verlag, Berlin, 1998.

\bibitem{BK} J. Brewer and L. Klingler, The ordered group of invertible ideals of a Pr\"{u}fer domain of finite character, {Comm. Alg.} {\bf 33} (2005), 4197--4203.

\bibitem{BY} H. S. Butts and R. W. Yeagy, Finite bases for integral closures, J. Reine Angew. Math {\bf 282} (1976), 114--125.

\bibitem{CC} P. J. Cahen and J. L. Chabert, {\it Integer-valued polynomials}, Mathematical Surveys and Monographs, 48. American Mathematical Society, Providence, RI, 1997.

\bibitem{CL} L. Claborn, Specified relations in the ideal group, {Michigan Math. J.} {\bf 15} (1968), 249--255.

\bibitem{CA} P. F. Conrad and D. McAlister, The completion of a lattice ordered group, {J. Australian Math. Soc.} {\bf 9} (1969), 182--208.

%\bibitem{EI} D. Eisenbud, {\it Commutative Algebra with a view towards Algebraic Geometry}, Graduate text in mathematics; Springer: New York, 1996; Vol. 150.

%\bibitem{EH} D. Eisenbud and J. Harris, {\it The Geometry of Schemes}, Graduate text in mathematics; Springer Verlag: New York, 2000; Vol. 197.

% \bibitem{EP} A.J. Engler and A. Prestel, {\it Valued Fields},
%Springer Monographs in Mathematics, May 2005.


%\bibitem{FHP} M. Fontana, J. Huckaba, and I. Papick, {\it
%Pr\"{u}fer Domains}, Marcel Dekker, New York, 1997.

%\bibitem{FL} M. Fontana and K.A. Loper, {\it Kronecker function rings: a genaral approach,} in ``Ideal Theoretic methods in Commutative Algebra'' (D.D.Anderson and I.J.Papick, Editors), M. Dekker Lecture Notes Pure Appl. Math. \textbf{220} (2001), 189-205.

%\bibitem{FL1} M. Fontana and K.A. Loper, Nagata rings, Kronecker function rings and related smistar operations, \textit{Comm. Algebra} \textbf{31} (2003), 4775-4805.

\bibitem{FL} M. Fontana and K. A. Loper, An historical overview of Kronecker function rings, Nagata rings, and related star and semistar operations, in {\it Multiplicative Ideal Theory in Commutative Algebra} (J.W. Brewer, S. Glaz, W. Heinzer and B. Olberding, Editors), Springer (2006), 169--187.

\bibitem{FHL} M. Fontana, E. Houston and T. Lucas,  Factoring ideals in integral domains. Lecture Notes of the Unione Matematica Italiana, 14. Springer, Heidelberg; UMI, Bologna, 2013.

\bibitem{FHP}
M. Fontana, J. Huckaba and I. Papick, {\it Pr\"ufer domains}. Monographs and Textbooks in Pure and Applied Mathematics, 203. Marcel Dekker, Inc., New York, 1997.

\bibitem{FO} R. M. Fossum, {\it The divisor class group of a Krull domain}, Ergebnisse der Mathematik und ihrer Grenzgebiete, Band 74, Springer-Verlag, New York-Heidelberg, 1973.

\bibitem{FS} L. Fuchs and L. Salce, {\it Modules over non-Noetherian domains}, Mathematical Surveys and Monographs 84, American Mathematical Society, Providence, RI, 2001.

\bibitem{GR} A. Grams, Atomic rings and the ascending chain condition for principal ideals, Proc. Cambridge Philos. Soc. {\bf 75} (1974), 321--329.

\bibitem{GO} R. Gilmer, Overrings of Pr\"ufer domains, J. Algebra {\bf 4} (1966), 331--340.

\bibitem{GH} R. Gilmer and  W. J. Heinzer, Overrings of Pr\"ufer domains II, J. Algebra {\bf 7} (1967), 281--302.

\bibitem{GG} R. Gilmer and A. Grams, The equality $(A\cap B)^n=A^n\cap B^n$ for ideals, Canad. J. Math. {\bf 24} (1972), 792--798.

\bibitem{Gil} R. Gilmer, {\it Multiplicative ideal theory.} Corrected reprint of the 1972 edition. \textit{Queen's Papers in Pure and Applied Mathematics}, 90. Queen's University, Kingston, ON, 1992.

\bibitem{Gil2} R. Gilmer, Pr\"ufer domains and rings of integer-valued polynomials, J. Algebra {\bf 129} (1990), no. 2, 502--517.

%\bibitem{GH} R. Gilmer and W. Heinzer, Irredundant intersections of
%valuation rings, \textit{Math. Zeitschr.} {\bf 103} (1968), 306-317.

\bibitem{GLE} A. M. Gleason, Projective topological spaces, Illinois. J. Math. {\bf 2} (1958), 482--489.

%\bibitem{HK} F. Halter-Koch, Kronecker Function Rings and Generalized
%Integral closures. \textit{Comm. Algebra}, vol. 31, No 1,
%45-59, 2003.

\bibitem{Has} R. Hasenauer, Normsets of almost Dedekind domains and atomicity, J. Commut. Alg., to appear.

%\bibitem{HGK} M. Hazewinkel, N. Gubareni and V. V. Kirichenko, {\it Algebras, rings and modules}, Lie algebras and Hopf algebras, Mathematical Surveys and Monographs, 168, American Mathematical Society, Providence, RI.

%\bibitem{HHP} W. Heinzer, J. A. Huckaba and I. J. Papick, $m$-canonical ideals in integral domains, \textit{Comm. Algebra}, Vol. 26, No 9, 3021-3043, 1998.

\bibitem{HO} W. Heinzer and B. Olberding, Unique irredundant intersections of completely irreducible ideals, J. Algebra {\bf 287} (2005), 432--448.

%\bibitem{HE} O. K. Heubo, Kronecker function rings of transcendental field extensions,
%\textit{Comm. Algebra}, to appear.

\bibitem{Heu} O. A. Heubo-Kwegna,  Kronecker function rings of transcendental field extensions, Comm. Algebra {\bf 38} (2010), no. 8, 2701--2719.

\bibitem{Joh} P. Johnstone, {\it Stone spaces}. Reprint of the 1982 edition. Cambridge Studies in Advanced Mathematics, 3. Cambridge University Press, Cambridge, 1986.

\bibitem{KM} M. Knox and W. W. McGovern, Rigid extension of $\ell$-groups of continuous functions, Czech. Math. J. {\bf 58} (2008), no. 4, 993--1014.

\bibitem{KO} S. Koppelberg, {\it Handbook of Boolean Algebras,} Vol. 1, North Holland, 1989.

%\bibitem{KP} F. V. Kulhmann and A. Prestel, On places of algebraic function fields.
%\textit{Reine Angew}. Math., 353:181-195, 1984.

\bibitem{Lel} G. Leloup, Preorders, rings, lattice-ordered groups and formal power series, in {\it Valuation theory and its applications II} (F.V. Kuhlmann, S. Kuhlmann, M. Marshall, editors), American Mathematical Society (2003) 157--174.

\bibitem{Lop} K. A. Loper, Almost Dedekind domains which are not Dedekind, in {\it Multiplicative Ideal Theory in Commutative Algebra} (J.W. Brewer, S. Glaz, W. Heinzer and B. Olberding, Editors), Springer (2006), 279--292.

\bibitem{LL} K. A. Loper and T. G. Lucas, Factoring ideals in almost Dedekind domains. {J. Reine Angew. Math.} {\bf 565} (2003), 61--78.

\bibitem{Mat} H. Matsumura, {\it Commutative ring theory,} Cambridge Studies in Advanced Mathematics 8, Cambridge University Press, 1986.

\bibitem{MS} S. Mazurkiewicz and W. Sierpi\'nski, Contribution \`a la topologie des ensembles d\'enombrables, Fund. Math. {\bf 1} (1920), 17--27.

\bibitem{WM} W. W. McGovern, B\'{e}zout SP-Domains, Comm. Alg. {\bf 35} (2007), 1777--1781.

\bibitem{Moc} J. Mo{\v{c}}ko{\v{r}}, {\it Groups of divisibility.} Mathematics and its Applications (East European Series). D. Reidel Publishing Co., Dordrecht, 1983.

%\bibitem{JM} J. Mott, The group of divisibility and its applications, in {\it Conference on Commutative Algebra}, Lecture Notes in Mathematics Vol 311, 1973, 194--208.

%\bibitem{N} M. Nagata, {\it Local rings}, Interscience Tracts in Pure and Applied Mathematics, No. 13 Interscience Publishers a division of John Wiley \& Sons, New York-London 1962, xiii+234 pp.

\bibitem{Ohm} J. Ohm, Semi-valuations and groups of divisibility. Canad. J. Math. {\bf 21} (1969), 576--591.

%\bibitem{OM} A. Okabe and R. Matsuda, Semistar-operations on integral domains. \textit{Math. J. Toyama Univ.}, 17:1-21, 1994.

%\bibitem{OM1} A. Okabe and R. Matsuda, Kronecker function rings of semistar-operations. \textit{Tsukuba J. Math.}, 21(2):529-540, 1997.

\bibitem{Olb} B. Olberding, Factorization into Radical Ideals, in {\it Arithmetical properties of commutative Rings and monoids} (S. Chapman, editor), Lect. Notes in Pure Appl. Math. Vol 241. Chapman \& Hall, 363--377, 2005.

\bibitem{OlbH} B. Olberding, Holomorphy rings in function fields, in
{\it Multiplicative ideal theory in commutative algebra}, 331-348,
Springer-Verlag, 2006.

\bibitem{OlbMat} B. Olberding, On Matlis domains and Pr\"ufer sections of Noetherian domains, in {\it Commutative algebra and its applications}, 321--332, Walter de Gruyter, Berlin, 2009.

\bibitem{PW} J. R. Porter and R. G. Woods, {\it Extensions and absolutes of Hausdorff spaces}, Springer, New York, 1988.

%\bibitem{JQ} J. Querre, Id\'{e}aux divisorielsd'un anneau de polyn\^{o}mes,
 % \textit{J. Algebra} {\bf 64} (1980), 270-284.

\bibitem{RE} A. Reinhart, Radical factorial monoids and domains, Ann. Sci. Math. Qu\'ebec {\bf 36} (2012), no. 1, 193--229.

\bibitem{REK} A. Reinhart, On monoids and domains whose monadic submonoids are Krull, in {\it Commutative Algebra - Recent Advances in Commutative Rings, Integer-valued Polynomials, and Polynomial Functions} (M. Fontana, S. Frisch, and S. Glaz, editors), Springer, 2014.

\bibitem{Rib} P. Ribenboim, Boolean powers, Fund. Math. {\bf 65} (1969), 243--268.

\bibitem{RY} W. Rump and Y. C. Yang, Jaffard-Ohm correspondence and Hochster duality, Bull. London Math. Soc. {\bf 40} (2008), no. 2, 263--273.

%\bibitem{SH} I. Swanson and C. Huneke , {\it Integral closure of Ideals, Rings, and Modules},
%London Mathematical Society Lecture Note Series. 336, Cambridge University Press, 2006.

\bibitem{VY} N. H. Vaughan and R. W. Yeagy, Factoring ideals into semiprime ideals, Canad. J. Math. {\bf 30} (1978), no. 6, 1313--1318.

\bibitem{Y} R. W. Yeagy, Semiprime factorizations in unions of Dedekind domains, J. Reine Angew. Math. {\bf 310} (1979), 182--186.

\bibitem{Zaf} M. Zafrullah, On generalized Dedekind domains, Mathematika {\bf 33} (1986), no. 2, 285--295.

%\bibitem{ZS} O. Zariski and P. Samuel, {\it Commutative algebra}. Vol. II.
%  Graduate Texts in Mathematics, Vol. 29. Springer-Verlag, New
%York-Heidelberg, 1975.

\end{thebibliography}
\end{document}